\documentclass[reqno,final]{amsart}

\usepackage[utf8]{inputenc} 
\usepackage[english]{babel}
\usepackage{mathabx}
\usepackage{amstext}
\usepackage{euscript}
\usepackage{bbm}
\usepackage{mathrsfs}
\usepackage{enumerate}
\usepackage{graphicx}
\usepackage{caption}
\usepackage{amscd}
\usepackage{verbatim}
\usepackage{amssymb}
\usepackage{amsthm}
\usepackage{amsmath}
\usepackage{amstext}
\usepackage{amsfonts}
\usepackage{latexsym}
\usepackage{mathtools} 
\usepackage{enumitem} 
\usepackage{accents} 
\usepackage[foot]{amsaddr} 

\usepackage[usenames,dvipsnames]{xcolor} 
\definecolor{dkblue}{RGB}{1,31,91} 
\usepackage[colorlinks=true, pdfstartview=FitV, linkcolor=dkblue, citecolor=dkblue, urlcolor=dkblue]{hyperref}

\usepackage{todonotes}

\usepackage[color,notref,notcite]{showkeys}

\usepackage{cite}

\newcommand{\newK}{\tilde{K}}

\newcommand{\qlep}{{|q|\le \frac{1}{2}|p|^{1/m}}}
\newcommand{\qgep}{{|q|\ge \frac{1}{2}|p|^{1/m}}}
\newcommand{\eqdef }{\overset{\mbox{\tiny{def}}}{=}}
\newcommand{\pv}{p}
\newcommand{\pZ}{\pv^0}
\newcommand{\qv}{q}
\newcommand{\qZ}{\qv^0}

\newcommand{\rth}{{\mathbb{R}^3}}
\newcommand{\rfo}{{\mathbb{R}^4}}

\newcommand{\singA}{a}
\newcommand{\singB}{b}
\newcommand{\singS}{\rho}
\newcommand{\zetaL}{\zeta_L}

\newcommand{\zetaTL}{\tilde{\zeta}_L}

\newcommand{\zetaLTone}{\tilde{\zeta}_{L}^1}

\newcommand{\zetaLTtwo}{\tilde{\zeta}_{L}^2}

\newcommand{\zetaTone}{\tilde{\zeta}_1}

\newcommand{\zetaTZ}{\tilde{\zeta}_{0}}
\newcommand{\zetaZ}{\zeta_{0}}

\newcommand{\zetaTZm}{\tilde{\zeta}_{0,m}}
\newcommand{\zetaZm}{{\zeta}_{0,m}}
\newcommand{\zetaTLm}{\tilde{\zeta}_{L,m}}
\newcommand{\zetaLm}{{\zeta}_{L,m}}

\theoremstyle{definition}
\newtheorem{theorem}{Theorem}

\newtheorem{lemma}[theorem]{Lemma}
\newtheorem{proposition}[theorem]{Proposition}
\newtheorem{remark}[theorem]{Remark}

\numberwithin{equation}{section}
\numberwithin{theorem}{section}

\begin{document}


%
\title[Frequency multiplier estimates]{Frequency multiplier estimates for the linearized relativistic Boltzmann operator without angular cutoff}

\author[J. W. Jang]{Jin Woo Jang$^\dagger$}
\address{$^\dagger$Institute for Applied Mathematics, University of Bonn, 53115 Bonn, Germany. \href{mailto:jangjinw@iam.uni-bonn.de}{jangjinw@iam.uni-bonn.de} }
\thanks{$^\dagger$Supported by the German DFG grant CRC 1060 and previously supported by the Korean IBS grant IBS-R003-D1 and partially by the NSF grant DMS-1500916 of the USA}

\author[R. M. Strain]{Robert M. Strain$^{\ddagger}$}
\address{$^\ddagger$Department of Mathematics, University of Pennsylvania, Philadelphia, PA 19104, USA. \href{mailto:strain@math.upenn.edu}{strain@math.upenn.edu}}
\thanks{$^\ddagger$Partially supported by the NSF grant DMS-1764177 of the USA}

\begin{abstract}
This paper is concerned with the relativistic Boltzmann equation without angular cutoff. The non-cutoff theory for the relativistic Boltzmann equation has been rarely studied even under a smallness assumption on the initial data due to the lack of understanding on the spectrum and the need for coercivity estimates on the linearized collision operator.  Namely, it is crucial to obtain the sharp asymptotics for the frequency multiplier to obtain this coercivity that has never been established before.  In this paper, we prove the sharp asymptotics for the frequency multiplier for a general relativistic scattering kernel without angular cutoff. As a consequence of our calculations, we further explain how the well known change of variables $p' \to p$ is not well defined in the special relativistic context. 
\end{abstract}

\thispagestyle{empty}

\maketitle
\tableofcontents

\section{Introduction}

The Boltzmann equation is a fundamental mathematical model for the dynamics of a gas represented as a collection of molecules. This equation was derived by Ludwig Boltzmann \cite{MR0158708} in 1872; this is a model for the collisional dynamics of a non-relativistic gas of particles.  For the collisional dynamics of a special relativistic gas, particles may have speeds that are comparable to the speed of light.  Lichnerowicz and Marrot \cite{MR0004796} have derived the relativistic Boltzmann equation in 1940. The relativistic Boltzmann equation has been a fundamental model for fast moving particles whose collisional dynamics are taken into account.

The special relativistic Boltzmann equation is given by
\begin{equation}
\label{RBE}
p^\mu\partial_\mu f={p^0}\partial_tf+cp\cdot\nabla_xf= C(f,f),
\end{equation}
where $c>0$ is the speed of light and the collision operator $C(f,f)$ can be written as
\begin{equation}
\label{Colop}
C(f,h)=\int_{\mathbb{R}^3}\frac{dq}{{q^0}}\int_{\mathbb{R}^3}\frac{dq'\  }{{q'^0}}\int_{\mathbb{R}^3}\frac{dp'\  }{{p'^0}} W(p,q|p',q')[f(p')h(q')-f(p)h(q)].
\end{equation}
Here, the transition rate $W(p,q|p',q')$ is defined as
\begin{equation}\label{transition.rate.RBE}
    W(p,q|p',q')=\frac{c}{2}s\sigma(g,\theta)\delta^{(4)}(p^\mu +q^\mu -p'^\mu -q'^\mu),
\end{equation}
where $\sigma(g,\theta)$ stands for the scattering kernel measuring the interactions between gas particles and the Dirac $\delta$-function expresses the conservation of energy and momentum for each collision.

This equation \eqref{RBE} is a relativistic generalization of the Newtonian Boltzmann equation:
\begin{equation}\label{newtonian.BE}
    \partial_tf +v\cdot\nabla_xf=\int_{\rth}dv_*\int_{\mathbb{S}^2}d\omega B(v-v_*,\omega)[f(v')f(v'_*)-f(v)f(v_*)],
\end{equation}
where the pre-post collisional velocities are
\begin{equation}\notag
    v'=\frac{v+v_*}{2}+\frac{|v-v_*|}{2}\omega, 
    \quad 
    v'_*=\frac{v+v_*}{2}-\frac{|v-v_*|}{2}\omega, \quad v, v_* \in \rth,
\end{equation}
and the collision kernel $B$ depends only on the relative velocity $|v-v_*|$ and the scattering angle $\omega\in \mathbb{S}^2$.
The mathematical analysis of the Boltzmann equation such as the well-posedness of the equation or the regularity of the solution crucially depends on the assumptions on the scattering kernel $B(v-v_*,\omega)$. 
The kernel $B$ is in general assumed to be in the form of a product in its arguments as 
$$B(v-v_*,\omega)=\Psi(|v-v_*|)b_0(\omega),$$ where both $\Psi$ and $b_0$ are assumed to be non-negative. This assumption is general and it includes the varied kinds of  collision kernels such as the hard-sphere collision kernel $\Psi(|v-v_*|)\approx |v-v_*|$, the collision kernel for Maxwellian molecules $\Psi(|v-v_*|)\approx 1$, the collision kernel for the  inverse-power law potential $\psi(r)=\frac{1}{r^{p-1}}$ with $B\approx |v-v_*|^{\gamma} \theta^{-\gamma'} b'_0(\theta)$ where $\gamma=\frac{p-5}{p-1}$, $\gamma'=\frac{2p}{p-1}$, $b'_0$ is bounded, and $\cos\theta =\frac{v-v_*}{|v-v_*|}\cdot \omega$,  and the assumption also includes many other kernels.

\subsection{Motivation}\label{sec:motivation}
Regarding the assumptions on the angular kernel $b_0(\omega)$, Grad \cite{MR0156656} proposed a ``cutoff" assumption. Namely, we say that the Boltzmann equation is in the ``cutoff" regime if $b_0(\omega)\in L^1(\mathbb{S}^2)$ or $b_0(\omega)\in L^\infty(\mathbb{S}^2)$.  Otherwise, we call it the Boltzmann equation without angular cutoff; sometimes this is called the ``non-cutoff'' regime.  This cut-off assumption is indeed very powerful in the mathematical analysis, as it removes the singularity from the angular kernel and allows one to split the gain and the loss terms of the Boltzmann operator.

However, it has been well-known that the regularity of a solution to the Boltzmann equation depends crucially on the assumption. For the angular kernel with the Grad cutoff, it has been known to propagate singularities \cite{MR1798557, MR2435186}. On the other hand, it has been known that the Boltzmann equation without angular cutoff has smoothing effects \cite{Li94,MR2679369,1909.12729,MR4195746}. In the case without angular cutoff, one has to make use of the cancellation between the gain and the loss terms to estimate the angular singularity.  Without angular cutoff the Boltzmann operator behaves as the fractional Laplacian on a lifted paraboloid of the energy-momentum four-vector \cite{MR2807092}.

The existence theory for the Boltzmann equation without angular cutoff was developed in the class of weak solutions via the method of renormalization \cite{MR1857879}.  Further the existence and uniqueness theory was developed using the energy method via linearization nearby Maxwellian equilibrium in \cite{MR2784329, MR2795331, MR2793203, MR2847536}.  This current paper is mainly concerned with the energy method nearby equilibrium.  One of the most crucial parts in the proof via the energy method is to create a positive dissipation term in the energy inequality. It turns out that the coercivity estimates for the dissipation term crucially depends on the asymptotics of the \textit{frequency multiplier} \eqref{tildezeta} whose explicit form will be introduced later on \eqref{def.zeta} and \eqref{zetaK.def}.

Regarding the Newtonian Boltzmann equation \eqref{newtonian.BE}, the estimates on the asymptotics of the \textit{frequency multiplier} have been proved by Pao \cite{Pao} using the symmetry of the linearized operator and using the sharp pointwise estimates of certain special functions.  This can also be proven using the procedure outlined in Section \ref{sec:main.difficult}. These asymptotics have been crucially used in the coercivity estimates and the spectral theory for the linearized Boltzmann operator in \cite{MR2322149}.  These coercivity estimates, and in addition the Newtonian cancellation lemma from \cite{ADVW} has been crucially used for the proof of the global in time wellposedness in \cite{MR2784329,MR2795331,MR2793203} nearby equilibrium.

In this paper, we are interested in proving analogous results for the relativistic Boltzmann equation \eqref{RBE}. Namely, we would like to establish the estimates on the asymptotics of the relativistic frequency multipliers for the linearized Boltzmann operator.  However, in the relativistic case, it turns out  that the collisional structure is substantially different \cite{ChapmanJangStrain2020}, and the crucial change of variable $p'\to p$ in the non-relativistic cancellation lemma does not hold in the relativistic case (which is explained in Remark \ref{cancel.problem}).  This also shows the major difficulty in the relativistic case versus  the non-relativistic case (in regards to the lack of the change of variables from $p' \to p$) and in regards to the inability to use the standard proof of the behavior of the frequency multiplier term $\zeta$ from the non-relativstic case. This difficulty will be revisited in Section \ref{sec:main.difficult} after we explain the explicit forms of the relativistic frequency multipliers.

Before we introduce our main theorem and strategies, we first introduce the notations, the hypothesis on the scattering kernel, and the linearization around equilibrium of the relativsitic Boltzmann equation.

\subsection{Notations}\label{Lor}
The relativistic momentum of a particle can be usefully denoted by a four-vector representation $p^\mu$ where $\mu=0,1,2,3$. Without loss of generality we normalize the mass of each particle $m=1$.
We raise and lower the indices with the Minkowski metric $p_\mu=\eta_{\mu\nu}p^\nu$, where the metric is defined as $\eta_{\mu\nu}=\text{diag}(-1, 1, 1, 1).$
The signature of the metric throughout this paper is $( - + + + )$.
With $p=(p_1,p_2,p_3)\in \rth$, we write $p^\mu=({p^0},p)$, where ${p^0}$ is the energy of a relativistic particle with momentum $p$.  The energy is defined as ${p^0}=\sqrt{c^2+|p|^2}$.
The product between the four-vectors with raised and lowered indices is the Lorentz inner product which is given by
$$
p^\mu q_\mu=-{p^0}{q^0}+\sum^3_{i=1} p_iq_i.
$$
Note that the momentum for each particle satisfies the mass shell condition $p^\mu p_\mu=-c^2$ with ${p^0}>0$. We also note that the product $p^\mu q_\mu$ is a Lorentz invariant.

By expanding the relativistic Boltzmann equation \eqref{RBE} and dividing both sides by ${p^0}$ we write the relativistic Boltzmann equation as
\begin{equation}
    \label{RBE.2}
\partial_t F+\hat{p}\cdot \nabla_x F= Q(F,F),
\end{equation}
where $Q(F,F)=C(F,F)/{p^0}$ is from \eqref{Colop} and \eqref{omegaint}.  The normalized velocity $\hat{p}$ is further given by
$$
\hat{p}=c\frac{p}{{p^0}}=\frac{p}{\sqrt{1+|p|^2/c^2}}.
$$ 
We will normalize the speed of light to be $c=1$ without loss of generality in the rest of this paper.

We also define ``$s$'' and ``$g$'' which respectively stand for the square of the energy and the relative momentum in the \textit{center-of-momentum} system, $p+q=0$, as
\begin{equation}
\label{s}
\begin{split}
s&\eqdef s(p^\mu,q^\mu)
=-(p^\mu+q^\mu)(p_\mu+q_\mu)=2(-p^\mu q_\mu+1)
\\
&=2(p^0q^0-p\cdot q +1)\geq 0.
\end{split}
\end{equation}
and
\begin{equation}
\label{g}
g\eqdef g(p^\mu,q^\mu)=\sqrt{(p^\mu-q^\mu)(p_\mu-q_\mu)} 
=\sqrt{2(p^0q^0-p\cdot q -1)}.
\end{equation}
Note that $s=g^2+4$.   
Now we record a useful sharp inequality for $g$:
\begin{equation}\label{g.ineq.sharp}
\frac{\sqrt{|p-q|^2+|p\times q|^2}}{\sqrt{p^0q^0}}\leq \sqrt{2\frac{|p-q|^2+|p\times q|^2}{ \pZ  \qZ +p\cdot q+1}} =  g   \leq |p-q|.
\end{equation}
See \cite{GS3} and Lemma \ref{lem:useful.ests} below. Similarly we will define $\bar{g}$ as the relative momentum between $p'^\mu$ and $p^\mu$. It is then written as
\begin{equation}
\label{gbar}
\begin{split}
\bar{g}&\eqdef g(p'^\mu,p^\mu)=\sqrt{(p'^\mu-p^\mu)(p'_\mu-p_\mu)}=\sqrt{2(-p'^\mu p_\mu-1)}\\&=\sqrt{2(p'^0p^0-p'\cdot p -1)}=\sqrt{2\frac{|p-p'|^2+|p\times p'|^2}{ p^0  p'^0 +p\cdot p'+1}}.\\
\end{split}
\end{equation}
In the same manner, we define the relative momentum between $p'^\mu$ and $q^\mu$ as
\begin{multline}\label{gtilde}
\tilde{g} \eqdef g(p'^\mu,q^\mu)=\sqrt{(p'^\mu-q^\mu)(p'_\mu-q_\mu)} =\sqrt{2(-p'^\mu q_\mu-1)}\\=\sqrt{2(p'^0q^0-p'\cdot q -1)}=\sqrt{2\frac{|p'-q|^2+|p'\times q|^2}{ p'^0  q^0 +p'\cdot q+1}}.
\end{multline}
From \eqref{transition.rate.RBE}, the conservation of energy and momentum for elastic collisions is described as
\begin{equation}\label{conservation}
p^\mu+q^\mu=p'^\mu+q'^\mu.
\end{equation}
The scattering angle $\theta$ is then defined by
\begin{equation}\notag 
\cos\theta=\frac{(p^\mu-q^\mu)(p'_\mu-q'_\mu)}{g^2}.
\end{equation}
Together with the conservation of energy and momentum as in \eqref{conservation}, it can be shown that the angle and $\cos\theta$ are well-defined \cite{MR1379589}. 

Note that the numerator of $\cos\theta$ can be further written as 
\begin{equation}
 \label{cos}
 \begin{split}
(p^\mu-q^\mu)(p'_\mu-q'_\mu)=&(p^\mu-q^\mu)(p_\mu+q_\mu-2q'_\mu)\\
=& (p^\mu-q^\mu)(p_\mu-q_\mu)+2 (p^\mu-q^\mu)(q_\mu-q'_\mu)\\
=&g^2+2(p^\mu-q^\mu)(q_\mu-q'_\mu)\\
=& g^2 +2(p^\mu-p'^\mu+p'^\mu-q^\mu)(p'_\mu-p_\mu)\\
=&g^2 -2(p'^\mu-p^\mu)(p'_\mu-p_\mu)+2(p'^\mu-q^\mu)(p'_\mu-p_\mu)\\
=&g^2 -2\bar{g}^2+2(p'^\mu-q^\mu)(p'_\mu-p_\mu)
=g^2-2\bar{g}^2.
 \end{split}
 \end{equation} 
 As above, we note that it follows from the collision geometry \eqref{conservation} that 
 $$
 (p'^\mu-q^\mu)(p'_\mu-p_\mu)=0.
 $$
 Therefore, using \eqref{cos} with \eqref{g} and \eqref{gbar} we can write
\begin{equation}\label{cosine.angle.formula}
    1-2\sin^2\frac{\theta}{2} =  \cos\theta=1-2\frac{\bar{g}^2}{g^2},
\end{equation}
 and hence we obtain that $\theta\approx \frac{\bar{g}}{g}$. This estimate will be used frequently in the rest of the paper.

\begin{remark}\label{angle.remark}
Since we are dealing with the non-cutoff relativistic Boltzmann equation then there will be an angular singularity when $\cos\theta=1$ as in \eqref{angassumption}.  The purpose of this remark is to explain the collisional geometry when $\cos\theta=1$.  
By \eqref{cosine.angle.formula} when $\cos\theta=1$ we have $\frac{\bar{g}^2}{g^2} =0$ which means 
$$
0=\bar{g}^2=(p^\mu-p'^\mu)(p_\mu-p'_\mu).
$$
Equivalently, this means that 
$$
({p'^0}-{p^0})^2=|p'-p|^2.
$$
And this implies that ${p^0}={p'^0}$ and $p=p'$ because 
$$
|{p'^0}-{p^0}| = \left|\frac{|p'|^2-|p|^2}{{p'^0}+{p^0}}\right|< |p'-p|.
$$
Therefore, if $\cos\theta=1$, we have $p'^\mu=p^\mu$ and also $q'^\mu=q^\mu$ by \eqref{conservation}. 
\end{remark}

Next we would like to introduce the relativistic Maxwellian which models the steady state solutions or equilibrium solutions also known as J\"uttner solutions.
These are characterized as a particle distribution which maximizes the entropy subject to constant mass, momentum, and energy. They are given by
$$
J(p)=\frac{e^{-\left(\frac{c}{k_BT}\right){p^0}}}{4\pi ck_BTK_2(\frac{c^2}{k_BT})},
$$
where $k_B$ is Boltzmann constant, $T$ is the temperature, and $K_2$ stands for the Bessel function $K_2(z)=\frac{z^2}{2}\int_1^\infty dt\ e^{-zt}(t^2-1)^\frac{3}{2}.$
  Then we obtain that the normalized  relativistic Maxwellian is given by
\begin{equation}\label{Juttner.equi}
J(p)=\frac{e^{-{p^0}}}{4\pi}.
\end{equation}
Throughout the rest of this paper, without loss of generality we will normalize all physical constants to 1 including the speed of light $c=1$.

We now consider the \textit{center-of-momentum} expression for the relativistic collision operator as below. Note that this expression has appeared in the physics literature; see \cite{MR635279}.
For other representations of the operator such as Glassey-Strauss coordinate expression, see \cite{MR1402446}, \cite{MR1105532}, and \cite{GS3}. Also, see \cite{MR2679588} for the relationship between those two representations of the collision operator.
As in \cite{MR635279}, one can reduce the collision operator (\ref{Colop}) using Lorentz transformations and get
\begin{equation}
\label{omegaint}
Q(f,h)
\eqdef
\frac{1}{p^0}C(f,h)
=
\int_\rth dq\int_{\mathbb{S}^2}d\omega\      v_{\text{\o}}\sigma(g,\theta)[f(p')h(q')-f(p)h(q)],
\end{equation}
where $v_{\text{\o}}=v_{\text{\o}}(p,q)$ is the M$\phi$ller velocity given by
\begin{equation}\notag 
v_{\text{\o}}(p,q)
\eqdef
\sqrt{\Big|\frac{p}{{p^0}}-\frac{q}{{q^0}}\Big|^2-\Big|\frac{p}{{p^0}}\times\frac{q}{{q^0}}\Big|^2}=\frac{g\sqrt{s}}{{p^0}{q^0}}.
\end{equation}
Comparing \eqref{Colop} with the reduced version of collision operator \eqref{omegaint}, we can notice that
one of the advantages of this \textit{center-of-momentum} expression of the collision operator is that the reduced integral (\ref{omegaint}) is written in relatively simple terms which only contains
the M$\phi$ller velocity, scattering kernel, and the cancellation between gain and loss terms.

The post-collisional momentum in the \textit{center-of-momentum} expression are written as
\begin{equation} \label{p'}
p'=\frac{p+q}{2}+\frac{g}{2}\Big(\omega+(\xi-1)(p+q)\frac{(p+q)\cdot\omega}{|p+q|^2}\Big),
\end{equation}
\begin{equation}\notag 
q'=\frac{p+q}{2}-\frac{g}{2}\Big(\omega+(\xi-1)(p+q)\frac{(p+q)\cdot\omega}{|p+q|^2}\Big),
\end{equation}
where $\xi\eqdef \frac{p^0+q^0}{\sqrt{s}}$; see for example \cite{MR2765751}.

We will frequently use the modified Bessel function of index zero given by
\begin{equation}\label{bessel0}
    I_0(y)\eqdef \frac{1}{2\pi}\int_0^{2\pi} d\phi \exp(y\cos\phi)
=\frac{1}{\pi}\int_0^{\pi} d\phi \exp(y\cos\phi),
\end{equation}
Throughout the paper, we use the notation $A \lesssim B$ if there exists a positive constant $c>0$ such that $A\le c B$ holds uniformly over the relevant parameters. Also, we use $A\approx B$ if both $A\lesssim B$ and $B\lesssim A$ hold.

We further remark that the decomposition of $q\in\rth $ into two regions $\qgep$ and $\qlep$ for a given $p$ will be crucial for our estimates in the rest of this paper.  Here we will choose some $m>1$ sufficiently large as explained in  Theorem \ref{main.thm}.  Then we will later show that the \textit{frequency multiplier} is lower order on the region $\qgep$.  To that end we introduce the following notation.  Given $h_1=h_1(p,q)$, we define the function $h=h(p)$ as an integral on $\rth $ as 
	$$h(p)=\int_\rth h_1(p,q)dq.$$ In this case, we split the integral into 
	$$\int_\rth = \int_\qgep +\int_\qlep,$$ and abuse the notations to denote each term as
	\begin{equation}\label{convention}\begin{split}[h]_\qgep &= h|_\qgep \eqdef \int_\qgep h_1(p,q)dq,\\
	[h]_\qlep &= h|_\qlep \eqdef  \int_\qlep h_1(p,q)dq.\end{split}\end{equation}
This convention of the notations will be used in a few convenient places in the rest of this paper.

\subsection{Main Hypothesis on the collision kernel $\sigma$}\label{hypo}
The relativistic Boltzmann collision kernel $\sigma(g,\theta)$ is a non-negative function which only depends on the relative velocity $g$ and the scattering angle
$\theta$. We assume that $\sigma$ takes the form of the product in its arguments; i.e., 
\begin{equation}\label{kernel.product}
    \sigma(g,\theta)\eqdef \Phi(g)\sigma_0(\theta).
\end{equation}
In general, we suppose that both $\Phi$ and $\sigma_0$ are non-negative functions.

Without loss of generality, we may assume that the collision kernel $\sigma$ is supported only when $\cos\theta\geq 0$ throught this paper;
i.e., $0\leq \theta \leq \frac{\pi}{2}$.
Otherwise, the following \textit{symmetrization} \cite{MR1379589} will reduce the case:
 \begin{equation}\notag 
\bar{\sigma}(g,\theta)=[\sigma(g,\theta)+\sigma(g,-\theta)]1_{\cos\theta\geq 0},
\end{equation}
where $1_A$ is the indicator function of the set $A$.

We suppose that the angular function $\theta \mapsto \sigma_0(\theta)$ is not locally integrable; for $C>0$, it satisfies
 \begin{equation}
 \label{angassumption}
\frac{1}{C\theta^{1+\gamma}} \leq \sin\theta\cdot\sigma_0(\theta) \leq \frac{C}{\theta^{1+\gamma}}, \hspace*{5mm}\gamma \in (0,2), \hspace*{5mm}\forall \theta \in \Big(0,\frac{\pi}{2}\Big].
\end{equation}
Notice that we do not assume any ``cut-off'' hypothesis on the angular function \cite{MR0156656}; we do not assume  that $\sigma_0\in L^1_{loc}(\mathbb{S}^2)$.  We further assume the collision kernel satisfies the following hard interactions assumption:
\begin{equation}
\label{hard}
 \Phi(g) =  C_{\Phi} g^{\singA},
 \quad - \gamma\leq {\singA} <2, \quad C_{\Phi}>0.
\end{equation}
Alternatively, we make the following soft interactions assumption
\begin{equation}
\label{soft}
 \Phi(g) =  C_{\Phi} g^{-\singB},
  \quad  \gamma < \singB < 2, \quad C_{\Phi}>0.
\end{equation}
These are the assumptions on the kernel that we will use throughout this paper. 
It will be convenient to introduce the following notation
\begin{equation}\label{rho.def}
\singS\eqdef    \begin{cases}
\singA&\text{under the hard interaction assumption \eqref{hard}}\\
-\singB& \text{under the soft interaction assumption \eqref{soft}}.
    \end{cases}
\end{equation}
These assumptions on the collision kernel have been motivated from important physical interactions. Conditions on our collision kernel are generic  in the sense of Dudy\'nski and Ekiel-Je$\dot{\text{z}}$ewska \cite{D-E3}; see also \cite{Dudynski2} for further physical discussions.  

Unfortunately, to the best of our knowledge, the relativistic Boltzmann equation has not been studied without the ``cut-off'' hypothesis. In this paper we study the relativistic Boltzmann equation without assuming the Grad's angular cut-off hypothesis with the goal of obtaining a better understanding of relativistic gases.

\subsection{Linearization and reformulation of the Boltzmann equation}

We will consider the linearization of the collision operator and the perturbation around the relativistic Maxwellian \eqref{Juttner.equi} equilibrium state
\begin{equation}
\label{pert}
F(t,x,p)=J(p)+\sqrt{J(p)}f(t,x,p).
\end{equation}
We linearize the relativistic Boltzmann equation around the relativistic Maxwellian equilibrium state (\ref{pert}). By expanding the equation \eqref{RBE.2}, we obtain that
\begin{equation}
\label{Linearized B}
\partial_t f+\hat{p}\cdot\nabla_x f+L(f)=\Gamma(f,f), \hspace{10mm} f(0,x,v)=f_0(x,v),
\end{equation}
where the linearized relativistic Boltzmann operator $L$ is given by
\begin{multline*}
L(f)\eqdef -J^{-1/2}Q(J,\sqrt{J}f)-J^{-1/2}Q(\sqrt{J}f,J)\\
= \int_{\rth}dq \int_{\mathbb{S}^2} d\omega\  v_{\text{\o}} \sigma(g,\omega)\Big(f(q)\sqrt{J(p)} 
 +f(p)\sqrt{J(q)}
  -f(q')\sqrt{J(p')}
 \\
-f(p')\sqrt{J(q')}\Big)\sqrt{J(q)},
\end{multline*}
and the bilinear operator $\Gamma$ is given by
\begin{equation}
\label{Gamma1}
\begin{split}
\Gamma(f,h)&\eqdef J^{-1/2}Q(\sqrt{J}f,\sqrt{J}h)\\
&=\int_{\rth}dq \int_{\mathbb{S}^2} d\omega\  v_{\text{\o}} \sigma(g,\theta)\sqrt{J(q)}(f(q')h(p')-f(q)h(p)).
\end{split}
\end{equation}
Then notice that we have
 $$
L(f)=-\Gamma(f,\sqrt{J})-\Gamma(\sqrt{J},f).
$$
We define the weight function $\tilde{\zeta}$ such that
 \begin{multline}
 \label{25}
 \Gamma(\sqrt{J},f)=\left(\int_{\rth}dq\int_{\mathbb{S}^2}d\omega\  v_{\text{\o}} \sigma(g,\theta)(f(p')-f(p))\sqrt{J(q')}\sqrt{J(q)}\right)\\-\tilde{\zeta}(p)f(p),
 \end{multline}
where then 
\begin{equation}\label{tildezeta}
	\tilde{\zeta}(p) \eqdef
	\int_{\rth}dq\int_{\mathbb{S}^2}d\omega\  v_{\text{\o}} \sigma(g,\theta)(\sqrt{J(q)}-\sqrt{J(q')})\sqrt{J(q)}.
\end{equation} 
We now call $\tilde{\zeta}(p)$ the \textit{frequency multiplier} of the linearized Boltzmann collision operator.

As we introduced in Section \ref{sec:motivation}, it is crucial to obtain the sharp asymptotic behavior of $\tilde{\zeta}(p)$ for the proof of the coercivity estimates of the linearized relativistic Boltzmann operator without angular cutoff, which will be used crucially for the proof of the global well-posedness of the relativistic Boltzmann equation without angular cutoff nearby the Maxwellian equilibrium \eqref{Juttner.equi}. This global well-posedness is studied in our paper  \cite{RelBolNoncut2020}.

\subsection{Main theorem}
The main goal in this paper is to prove the following sharp leading order asymptotics on the \textit{frequency multiplier} \eqref{tildezeta}. 

\begin{theorem}\label{main.thm}Suppose that the angular function $\sigma_0$ satisfies the \textit{non-cutoff} assumption  \eqref{angassumption} under either \eqref{hard} or \eqref{soft}. Then the \textit{frequency multiplier} $\tilde{\zeta}(p)$ from \eqref{tildezeta} can be split into the sum of two frequency multiplier functions as 
$$\tilde{\zeta}=\zeta+\zeta_{\mathcal{K}},$$ 
which are defined in \eqref{def.zeta} and \eqref{zetaK.def} respectively. For both hard \eqref{hard} and soft \eqref{soft} interactions, using the notation \eqref{rho.def}, these multiplier functions satisfy the following asymptotics: 
\begin{equation}
\label{Paos}
\begin{split}
&|\zeta_\mathcal{K}(p)|\leq C_\varepsilon {(p^0)}^{\frac{\singS}{2}+\varepsilon}
\hspace{5mm}\text{and}\hspace{5mm}
\zeta(p)\approx {(p^0)}^{\frac{{\singS+\gamma}}{2}}.
\end{split}
\end{equation}
Above, for any small $\varepsilon>0$ there exists a finite constant $C_\epsilon>0$ as above.
\end{theorem}

The proof of our main Theorem \ref{main.thm} follows from the decomposition into \eqref{def.zeta} and \eqref{zetaK.def} of $\tilde{\zeta}(p)$ in \eqref{tildezeta} and the estimates in Proposition's \ref{prop.coercive}, \ref{prop.zeta0.asymptotic}, \ref{prop.zetaL.asymptotic}, and \ref{prop.zetaL.asymptoticnew}.  Our strategy of proof will be explained fully in Section's \ref{sec:main.decomp} and \ref{sec:proof.outline}.

\begin{remark}\label{leading.order.remark}
Throughout the paper, based upon \eqref{Paos}, we call a term $A(p)$ a \textit{leading order} term, if $A(p)\approx (p^0)^{\frac{\singS+\gamma}{2}}.$  In addition, we call a term $B(p)$ a \textit{lower order} term if  for some positive constant $\epsilon>0$ there exists a finite constant $C_\epsilon>0$ such that $|B(p)|\leq C_\epsilon (p^0)^{\frac{\singS+\gamma}{2}-\epsilon}$. 
\end{remark}

\subsection{Main difficulties in the relativistic case}\label{sec:main.difficult}
In this subsection, we sketch the main difficulty that we experience in the relativistic case versus in the non-relativistic case.  In contrast to \cite{Pao}, one can prove the asymptotic behavior of the non-relativistic frequency multiplier in the following simple way.

In the non-relativisitic case \eqref{newtonian.BE} the analog of the collision frequency multiplier \eqref{tildezeta} is given  \cite[Page 11]{MR2784329} by
\begin{equation}\label{newtonian.BE.nu}
\tilde{\nu}(v)=\int_{\rth}dv_*\int_{\mathbb{S}^2}d\omega\ B(v-v_*,\omega)
\left(\sqrt{\mu(v_*)}-\sqrt{\mu(v'_*)}\right)\sqrt{\mu(v_*)},
\end{equation}
where the Newtonian Maxwellian equilibrium is given by $$\mu(v)\eqdef (2\pi)^{-3/2} \exp\left(-|v|^2/2\right).$$  Note the similarity to  $\tilde{\zeta}$ in \eqref{tildezeta}.  In the non-relativistic case, due to symmetry the following decomposition of the frequency multiplier is very useful:
$$
\left(\sqrt{\mu(v_*)}-\sqrt{\mu(v'_*)}\right)\sqrt{\mu(v_*)}
=    
\frac{1}{2}\left(\sqrt{\mu(v_*)}-\sqrt{\mu(v'_*)}\right)^2 
+ \frac{1}{2} \left(\mu(v_*)-\mu(v'_*)\right).
$$ 
This decomposition allows the splitting 
$$
\tilde{\nu}(v) = \nu(v) + \nu_{\mathcal{K}}(v),
$$
where 
\begin{equation}\notag
\nu(v)\eqdef
\frac{1}{2}
\int_{\rth}dv_*\int_{\mathbb{S}^2}d\omega\  B(v-v_*,\omega)
\left(\sqrt{\mu(v_*)}-\sqrt{\mu(v'_*)}\right)^2.
\end{equation}
Now $\nu(v)$ is clearly non-negative and it can be quickly shown that $\nu(v)$ has the expected leading order asymptotic behavior as $|v|\to\infty$.  

On the other hand from \eqref{newtonian.BE.nu} we have
\begin{equation}\notag
\nu_{\mathcal{K}}(v) = 
\frac{1}{2}
\int_{\rth}dv_*\int_{\mathbb{S}^2}d\omega\  B(v-v_*,\omega)
\left(\mu(v_*)-\mu(v'_*)\right).
\end{equation}
Now one can  use the Newtonian cancellation lemma \cite{ADVW}, the change of variable from $v_*' \to v_*$, to show for some $C'>0$ that
\begin{equation}\notag
\nu_{\mathcal{K}}(v) = 
C'
\int_{\rth}dv_*\ \mu(v_*) \Psi(|v-v_*| ).
\end{equation}
This expression directly implies that $\nu_{\mathcal{K}}(v)$ has lower order asymptotic behavior as $|v|\to\infty$.  This decomposition is crucial to designing a norm that captures the sharp behavior of the linearized collision operator and to further prove the global in time existence of solutions nearby equilibrium.

Unfortunately in the relativistic case this approach fails as we now explain.  We recall that the main difference in the integrand of $\tilde{\zeta}$ in \eqref{tildezeta} is $$ \sqrt{J(q)}(\sqrt{J(q)}-\sqrt{J(q')}).$$
The analogous splitting in the relativistic case is
\begin{equation}\label{rel.difference}
    (\sqrt{J(q)}-\sqrt{J(q')})  \sqrt{J(q)}
=    
\frac{1}{2}\left(\sqrt{J(q)}-\sqrt{J(q')}\right)^2 
+ \frac{1}{2} \left(J(q)-J(q')\right).
\end{equation}
However, this decomposition does not help in the relativistic case and it is also closely related to the fact that the crucial change of variables $q'\to q$ in the cancellation lemma \cite{ADVW} is problematic in the relativistic case \cite{ChapmanJangStrain2020} even in the case with an angular cutoff.  We now provide the sketch of the argument.

We ignore the positive square term, $\left(\sqrt{J(q)}-\sqrt{J(q')}\right)^2$, and focus on the second term on the right side in \eqref{rel.difference}. We can write this term from \eqref{tildezeta} as
\begin{equation}\notag
	\tilde{\zeta}^B(p) \eqdef
	\frac{1}{2}\int_{\rth}dq\int_{\mathbb{S}^2}d\omega\  v_{\text{\o}} \sigma(g,\theta)(J(q)-J(q')) = \tilde{\zeta}^B_1(p) - \tilde{\zeta}^B_2(p).
\end{equation} 
We can further assume that the angular kernel $\sigma_0$ from \eqref{kernel.product} is just pointwise bounded (we do not need to assume that it is mean-zero) with an angular cutoff.  Then the term on the right side of \eqref{rel.difference} containing $J(q)$ in \eqref{tildezeta} corresponds to \eqref{Ilossrepresentation} in Appendix \ref{sec:alternative.deriv}.  Then \eqref{Ilossrepresentation} is transformed into \eqref{final loss} so that
\begin{multline}\notag
\tilde{\zeta}^B_1(p)=\frac{c'}{p^0}e^{\frac{p^0}{2}}\int_{\rth}\frac{dq}{q^0}\frac{e^{-\frac{1}{2}q^0}}{g} \int_{0}^{\infty} \frac{rdr}{\sqrt{r^2+s}}s_\Lambda\sigma(g_\Lambda,\theta_\Lambda)\\\times \exp\left(-\frac{p^0+q^0}{2\sqrt{s}}\sqrt{r^2+s}\right)I_0\left(\frac{|p\times q|}{g\sqrt{s}}r\right).
\end{multline}
Above $c'>0$, and we further use the notations \eqref{g2y.variable} below with $r=y\sqrt{s}$ in addition to the Bessel function \eqref{bessel0}.  We point out that the above is a finite integral since it contains exponential decay in both the $q$ and the $r$ variables in \eqref{final loss}.  However, if we look at the new loss term with $J(q')$ only, then if we follow the same derivation, then we obtain \eqref{final loss} without the exponential term and without the bessel function $I_0$.  Indeed following the transformation procedure in Appendix \ref{sec:alternative.deriv} we obtain
\begin{equation}\notag
\tilde{\zeta}^B_2(p)=\frac{c'}{p^0}\int_{\rth}\frac{dq}{q^0}\frac{e^{-q^0}}{g} \int_{0}^{\infty} \frac{rdr}{\sqrt{r^2+s}}s_\Lambda\sigma(g_\Lambda,\theta_\Lambda).
\end{equation}
Thus we see following this procedure that the $dr$ integration becomes infinite in $\tilde{\zeta}^B_2(p)$, since this term no longer contains sufficient decay in the $r$ variable.  Therefore, the whole integral becomes infinity unless we artificially assume that the kernel $\sigma$ decays very rapidly for large $r$.  (We mention that this can be directly justified using standard approximation procedures.)

\begin{remark}\label{cancel.problem}
This illustrates that the crucial change of variable in cancellation lemma, $q' \to q$ (or $(q', \omega) \to (q, k)$ for some $|k|=1$) as stated in \cite{ADVW} for the Newtonian Boltzmann equation, is not well defined in the special relativistic case.  This statement is further independent of our choice of coordinate representations of $(p', q')$, such as for example \eqref{p'}.\footnote{We refer to \cite{MR2765751} for a discussion of a variety of coordinate representations of the relativistic Boltzmann collision operator.}
\end{remark}

If the cancellation lemma were true then the integrand in $\tilde{\zeta}^B_2(p)$ would be integrable and $\tilde{\zeta}^B_2(p)$ would be finite.  Indeed, the expression $\tilde{\zeta}^B_2(p)$ integrates to infinity by this argument.  The factor $J(q')$ does not provide sufficient decay to control the integral. That is why such a decomposition, which was very effective in the Newtonian case, does not help in the relativistic case. It is adding and subtracting a term which is infinity.

\subsection{Main decomposition}\label{sec:main.decomp}
Instead we perform in Appendix \ref{sec:derivation} the following transformation of \eqref{tildezeta} as $\tilde{\zeta}=\zeta_0+\zetaL$  where for $c'>0$ we have
\begin{multline}\label{zeta0By}
\zeta_0 \eqdef \frac{c'}{p^0}\int_{\rth}\frac{dq}{q^0}\frac{e^{-q^0}\sqrt{s}}{g}\int_{0}^\infty \frac{ydy}{\sqrt{y^2+1}}s_\Lambda\sigma(g_\Lambda,\theta_\Lambda) \frac{s\Phi(g)g^4}{s_\Lambda  \Phi(g_\Lambda)g^4_\Lambda}
\\
\times
\Big[1 - \exp\left(l(1-\sqrt{y^2+1})\right)I_0\left(jy\right) \Big],
\end{multline}and
\begin{multline}\label{zetaLBy}
\zetaL \eqdef \frac{c'}{p^0}\int_{\rth}\frac{dq}{q^0}\frac{e^{-q^0}\sqrt{s}}{g}\int_{0}^\infty \frac{ydy}{\sqrt{y^2+1}}s_\Lambda\sigma(g_\Lambda,\theta_\Lambda)   \\
\times
\exp\left(l(1-\sqrt{y^2+1})\right) I_0\left(jy\right)
\left( \frac{s\Phi(g)g^4}{s_\Lambda  \Phi(g_\Lambda)g^4_\Lambda} -1\right).
\end{multline}
We recall the modified Bessel function from \eqref{bessel0}.
These expressions arise by applying the change of variables $r\mapsto y=\frac{r}{\sqrt{s}}$, to the expressions \eqref{zeta0B.appendix} and \eqref{zetaLB}.  Above we use the notations $l$ and $j$ that are defined as
\begin{equation}\label{lj}
l=l(p,q)\eqdef \frac{p^0+q^0}{4},\ \text{and}\ j=j(p,q)\eqdef \frac{|p\times q|}{2g}.
\end{equation}
We also further define the notations
\begin{equation}\label{g2y.variable}
    g^2_\Lambda = g^2+\frac{s}{2}(\sqrt{y^2+1}-1), \quad s_\Lambda = g^2_\Lambda + 4,
\end{equation}
and from \eqref{newcos} we have 
\begin{equation}\notag 
\cos\theta_\Lambda=\frac{2g^2}{g^2_\Lambda}-1
=
\frac{g^2-\frac{s}{2}(\sqrt{y^2+1}-1)}{g^2+\frac{s}{2}(\sqrt{y^2+1}-1)}.
\end{equation}
It is shown in Section \ref{sec:low order zeta1} that $\zetaL(p)$ in \eqref{zetaLBy} has lower order asymptotic behavior, and in Section \ref{sec:fullsharpupper zeta0} we see that the main part of $\zeta_0(p)$ in \eqref{zeta0By} has a leading order upper bound.

We remark that the dynamics of each decomposed piece of $\tilde{\zeta}$ in \eqref{zeta0By} and \eqref{zetaLBy}  are essentially depending on the integral domain for the $q$ and $y$ variables, and the Bessel function $I_0(jy)$ from \eqref{bessel0}, which makes them extremely complex.  It turns out that the major difficulty involves the difference inside the inside the integrand in \eqref{zeta0By}:
\begin{equation}\notag
    \Big[1 - \exp\left(l(1-\sqrt{y^2+1})+jy \cos\phi\right)\Big].
\end{equation}
This expression is zero at $y=0$ and converges to one as $y\to\infty$.  However this expression also 
has it's negative minimum at 
\begin{equation}\label{y.min.difference}
y_{m} = \frac{j|\cos\phi|}{\sqrt{l^2 - j^2 \cos^2 \phi}},
\end{equation}
and the difference above remains negative, for say $\cos\phi>0$, until
\begin{equation}\label{y.zero.change}
y_{s} = \frac{2l j \cos\phi}{l^2 - j^2 \cos^2 \phi},
\end{equation}
where $0 \le y_{m} \le y_{s} \to 0$ as $l \to \infty$.  Thus the whole integral in \eqref{zeta0By} is negative in a large region nearby the minimum of the difference and nearby the singularity of the kernel $\sigma(g_\Lambda, \theta_\Lambda)$ at $y=0$.     This region where the integrand is negative and close to it's minimum makes it extremely problematic to prove the required leading order asymptotic positive lower bound.  Therefore, it is unclear from this point of view whether or not \eqref{tildezeta} or  \eqref{zeta0By} is positive or a leading order term.  It is essential to have a positive leading order term for the collision frequency multiplier in order to obtain the sharp behavior of the linearized collision operator.

For this reason we needed to derive another representation of $\tilde{\zeta}$ from \eqref{tildezeta}.   As in \eqref{eq:tildezetanew1.appendix}, we can alternatively write $\tilde{\zeta}$ as
\begin{multline}\notag
\tilde{\zeta}(p)=\frac{c'}{\pi p^0}\int_{\rth}\frac{dq}{q^0}\frac{e^{-q^0}\sqrt{s}}{g}\int_{0}^\infty \frac{ydy}{\sqrt{y^2+1}}s_\Lambda\sigma(g_\Lambda,\theta_\Lambda)\int_0^{\pi} d\phi\\\times\Big[\exp (2l-2l\sqrt{y^2+1}+2jy\cos\phi)-\exp(l-l\sqrt{y^2+1} +jy\cos\phi)\Big],
\end{multline}
where we use the notations \eqref{lj} and \eqref{g2y.variable} and $c'>0$.
Then $\tilde{\zeta}$ can be decomposed  into two terms $\tilde{\zeta}=\tilde{\zeta}_0+\tilde{\zeta}_L$ as 
\begin{multline}\label{eq:tildezetanew1}
\tilde{\zeta}_0(p)=\frac{c'}{\pi p^0}\int_{\rth}\frac{dq}{q^0}\frac{e^{-q^0}\sqrt{s}}{g}\int_{0}^\infty \frac{ydy}{\sqrt{y^2+1}}s_\Lambda\sigma(g_\Lambda,\theta_\Lambda)\int_0^{\pi} d\phi\frac{s\Phi(g)g^4}{s_\Lambda  \Phi(g_\Lambda)g^4_\Lambda}  \\\times\Big[\exp (2l-2l\sqrt{y^2+1}+2jy\cos\phi)-\exp(l-l\sqrt{y^2+1} +jy\cos\phi)\Big],
\end{multline}
and
\begin{multline}\label{eq:tildezetanew1L}
\tilde{\zeta}_L(p)
=\frac{c'}{ p^0}\int_{\rth}\frac{dq}{q^0}\frac{e^{-q^0}\sqrt{s}}{g}\int_{0}^\infty \frac{ydy}{\sqrt{y^2+1}}s_\Lambda\sigma(g_\Lambda,\theta_\Lambda) \left(1-\frac{s\Phi(g)g^4}{s_\Lambda  \Phi(g_\Lambda)g^4_\Lambda}  \right) \\\times\Big[\exp (2l-2l\sqrt{y^2+1})I_0(2jy)-\exp(l-l\sqrt{y^2+1})I_0(jy)\Big].
\end{multline}
The expression \eqref{eq:tildezetanew1} looks like a better candidate for the leading order term because the difference 
$$
\Big[\exp (2l-2l\sqrt{y^2+1}+2jy\cos\phi)-\exp(l-l\sqrt{y^2+1} +jy\cos\phi)\Big]
$$
now has it's maximum, $y_m$, at \eqref{y.min.difference} and it is positive on $0 \le y \le y_s$ from \eqref{y.zero.change}.  However this difference is negative for $y \ge y_s$ and still remains large and negative in a significantly sized region after $y$ passes $y_s$ which causes further extreme difficulties in proving a leading order positive lower bound.

Instead the key point will be that when we add $\zeta_0$ and $\tilde{\zeta}_0$ then the resulting term is clearly positive and leading order.  And we will later show that $\tilde{\zeta}_L$ from \eqref{eq:tildezetanew1L} and $\zeta_L$ from \eqref{zetaLBy} are lower order terms.  In particular we can take the summation of $\frac{1}{2}\zetaZ(p)1_{|p|\ge 1}$ from \eqref{zeta0By} and $\frac{1}{2}\zetaTZ(p)1_{|p|\ge 1}$ from \eqref{eq:tildezetanew1}. By adding an additional term $\langle p \rangle^{(\rho+\gamma)/2} 1_{|p|\le 1}$,
we obtain that
\begin{multline}\label{def.zeta.mod}
\zeta'(p) \eqdef \left[\frac{\zetaZ(p)+\zetaTZ(p)}{2}\right]1_{|p|\ge 1}
+\langle p \rangle^{(\rho+\gamma)/2} 1_{|p|\le 1}\\
=
\frac{c'1_{|p|\ge 1}}{2\pi p^0}\int_{\rth}\frac{dq}{q^0}\frac{e^{-q^0}\sqrt{s}}{g}\int_{0}^\infty \frac{ydy}{\sqrt{y^2+1}}s_\Lambda\sigma(g_\Lambda,\theta_\Lambda)\frac{s\Phi(g)g^4}{s_\Lambda  \Phi(g_\Lambda)g^4_\Lambda}\int_0^{\pi} d\phi\\\times\Big[\exp (2l-2l\sqrt{y^2+1}+2jy\cos\phi)-2\exp(l-l\sqrt{y^2+1} +jy\cos\phi)+1\Big]\\+\langle p \rangle^{(\rho+\gamma)/2} 1_{|p|\le 1}\\
=
\frac{c'1_{|p|\ge 1}}{2\pi p^0}\int_{\rth}\frac{dq}{q^0}\frac{e^{-q^0}\sqrt{s}}{g}\int_{0}^\infty \frac{ydy}{\sqrt{y^2+1}}s_\Lambda\sigma(g_\Lambda,\theta_\Lambda)\frac{s\Phi(g)g^4}{s_\Lambda  \Phi(g_\Lambda)g^4_\Lambda}
\\
\times \int_0^{\pi} d\phi \Big[\exp(l-l\sqrt{y^2+1} +jy\cos\phi)-1\Big]^2+\langle p \rangle^{(\rho+\gamma)/2} 1_{|p|\le 1}.
\end{multline}
Note that $\zeta'(p)$ is then automatically positive for all $p$.  The term $\langle p \rangle^{(\rho+\gamma)/2} 1_{|p|\le 1}$ guarantees that $\zeta(p)>0$ near $p=0$ (if necessary).

However, unfortunately it turns out that there is an additional severe difficulty to obtain the sharp pointwise asymptotic upper bound of the expression \eqref{def.zeta.mod} for $\zeta(p)$.  The problem is that the following term 
\begin{equation}\notag
    \exp({-q^0})\exp (2l-2l\sqrt{y^2+1}+2jy\cos\phi)
\end{equation}
does not in general have uniform decay in the $q^0$ variable when we are close to the singularity in the $dy$ integral at $y=0$.  Here we recall the definitions \eqref{lj}.  Therefore while the expression $\zeta'(p)$ in \eqref{def.zeta.mod} appears good for obtaining a positive asymptotic lower bound, it is extremely difficult to obtain the required sharp upper bound.  There is the same difficulty for $\zetaTL(p)$ in \eqref{eq:tildezetanew1L}.  To overcome this situation we split the $dq$ integral into the two regions $\qgep$ and $\qlep$.  

\begin{remark}
 The derivations of the alternative formula's in \eqref{zeta0By}, \eqref{zetaLBy}, \eqref{eq:tildezetanew1} and \eqref{eq:tildezetanew1L} of \eqref{tildezeta} in Appendix \ref{sec:derivation} still hold under the restrictions to the region such as $\left[\tilde{\zeta}\right]_{|q|\ge \frac{1}{2}|p|^{1/m}}$ and $\left[\tilde{\zeta}\right]_{|q|\le \frac{1}{2}|p|^{1/m}}$ using the convention \eqref{convention}.  This is straightforward from the proofs in  Appendix \ref{sec:derivation}.
 \end{remark}

With these splittings, we are then able to obtain the sharp asymptotic upper bound estimates of $\zeta(p)$ and $\zetaTL(p)$ on the region $\qlep$.   And on $\qgep$ from \eqref{tildezeta} using \eqref{convention} we further define
 \begin{equation}\label{def.zetatilde1}
     \zetaTone(p) \eqdef \left[\tilde{\zeta}(p)\right]_{|q|\ge \frac{1}{2}|p|^{1/m}}1_{|p|\ge 1}
 \end{equation}
We will prove that $\zetaTone(p)$ has low order asymptotic behavior in Proposition \ref{prop.tildezeta.upper.qgep}.  Then these estimates together will be enough to prove our main theorem.

To this end, from \eqref{zeta0By} and \eqref{zetaLBy}, we also define the following two terms
\begin{equation}\label{def.zetaMdef}
    \zeta_{0,m}(p) \eqdef [\zeta_0(p)]_{|q| \le \frac{1}{2} |p|^{1/m}} 1_{|p|\ge 1},
    \quad 
   \zeta_{L,m}(p) \eqdef [\zetaL(p)]_{|q| \le \frac{1}{2} |p|^{1/m}} 1_{|p|\ge 1},   
\end{equation}
In \eqref{def.zetaMdef} above we use the notation convention from \eqref{convention}.  For example
\begin{multline}\label{def.convention.zeta0By}
\zeta_{0,m}(p) = \frac{c'}{p^0}\int_{|q| \le \frac{1}{2} |p|^{1/m}}\frac{dq}{q^0}\frac{e^{-q^0}\sqrt{s}}{g}\int_{0}^\infty \frac{ydy}{\sqrt{y^2+1}}s_\Lambda\sigma(g_\Lambda,\theta_\Lambda) \frac{s\Phi(g)g^4}{s_\Lambda  \Phi(g_\Lambda)g^4_\Lambda}
\\
\times
\Big[1 - \exp\left(l(1-\sqrt{y^2+1})\right)I_0\left(jy\right) \Big],
\end{multline}
Similarly, from \eqref{eq:tildezetanew1} and \eqref{eq:tildezetanew1L}, we further define 
\begin{equation}\label{def.TzetaMdef}
    \zetaTZm(p) \eqdef [\tilde{\zeta}_0(p)]_{|q| \le \frac{1}{2} |p|^{1/m}} 1_{|p|\ge 1},
    \quad 
   \zetaTLm(p) \eqdef [\tilde{\zeta}_L(p)]_{|q| \le \frac{1}{2} |p|^{1/m}} 1_{|p|\ge 1},   
\end{equation}
With these definitions instead of \eqref{def.zeta.mod} we define the modified frequency multiplier
\begin{multline}\label{def.zeta}
\zeta(p) \eqdef \frac{1}{2} \left[\zeta_{0,m}(p)+\zetaTZm(p) \right]
+\langle p \rangle^{(\rho+\gamma)/2} 1_{|p|\le 1}\\
=
\frac{c'1_{|p|\ge 1}}{2\pi p^0}\int_{\qlep}\frac{dq}{q^0}\frac{e^{-q^0}\sqrt{s}}{g}\int_{0}^\infty \frac{ydy}{\sqrt{y^2+1}}s_\Lambda\sigma(g_\Lambda,\theta_\Lambda)\frac{s\Phi(g)g^4}{s_\Lambda  \Phi(g_\Lambda)g^4_\Lambda}
\\
\times \int_0^{\pi} d\phi \Big[\exp(l-l\sqrt{y^2+1} +jy\cos\phi)-1\Big]^2+\langle p \rangle^{(\rho+\gamma)/2} 1_{|p|\le 1}.
\end{multline}
Again this is automatically positive.

As such, we introduce the construction of a positive leading order term and a lower order term in the frequency multiplier $\tilde{\zeta}(p)$ from \eqref{tildezeta}, which is highly non-trivial.   We suggest the following novel decomposition of $\tilde{\zeta}=\zeta+\zeta_{\mathcal{K}}$ as
$$
\tilde{\zeta}(p)=\zeta(p)+\zeta_{\mathcal{K}}(p)
$$ 
where $\zeta$ is given by \eqref{def.zeta}, using also \eqref{tildezeta}, \eqref{zetaLBy}, \eqref{eq:tildezetanew1L} and \eqref{def.zetatilde1} we have
\begin{equation}\label{zetaK.def}
    \zeta_{\mathcal{K}}\eqdef \tilde{\zeta}(p)1_{|p|\le 1}
    +
     \zetaTone(p)
     +
     \frac{1}{2}\left(\zetaLm+\zetaTLm\right) 
     - \langle p \rangle^{(\rho+\gamma)/2} 1_{|p|\le 1},
\end{equation}
 Then we will show that $\zeta$ and $\zeta_{\mathcal{K}}$ satisfy the asymptotics from \eqref{Paos}.   In particular the main positive term is \eqref{def.zeta}.

\subsection{Outline of the proof of the main theorem}\label{sec:proof.outline}
 Specifically, in the rest of the paper, in order to prove Theorem \ref{main.thm} we will make upper- and lower-bound estimates for  $\zeta$ in \eqref{def.zeta} and will conclude that it is a leading order term. In addition, we will show that $\zeta_{\mathcal{K}}$ in \eqref{zetaK.def} is a lower-order term.

 We will first prove that $\zeta(p)$ from \eqref{def.zeta} has a leading order positive lower bound in Proposition \ref{prop.coercive}.  Then we will prove that $\zeta_0$  from \eqref{zeta0By} has the leading order upper bound in Proposition \ref{prop.zeta0.asymptotic}.  Then in Proposition \ref{prop.zetaL.asymptotic} we prove that $\zetaL(p)$ from \eqref{zetaLBy} has a lower order upper bound.  We further prove in Proposition \ref{prop.zetaL.asymptoticnew} that $\zetaTLm(p)$ from \eqref{def.TzetaMdef} with \eqref{eq:tildezetanew1L} has a lower order upper bound. We then prove in Proposition \ref{prop.tildezeta.upper.qgep} that $\zetaTone(p)$ from \eqref{def.zetatilde1} has a lower order upper bound. Note that both $\tilde{\zeta}(p)1_{|p|\le 1}$ and $\langle p \rangle^{(\rho+\gamma)/2} 1_{|p|\le 1}$ have lower order upper bounds since  $1_{|p|\le 1} $ trivially makes $p^0\lesssim 1.$   All of these estimates combine to prove that $\zeta_{\mathcal{K}}$ has a lower order upper bound, and that $\tilde{\zeta}(p)$ from \eqref{tildezeta} has a leading order asymptotic upper bound.

We remark that we have not estimated the asymptotic upper bound of $\zetaTZ(p)$ from \eqref{eq:tildezetanew1} or more accurately we have not estimated $\zetaTZm(p)$ from \eqref{def.TzetaMdef} and this is not necessary because from the splittings above we have
 \begin{equation}\notag
     \zetaTZm(p) = [\tilde{\zeta}(p)]_{\qlep}1_{|p|\ge 1} - \zetaTLm(p)
    = \zetaZm(p) + \zetaLm(p)- \zetaTLm(p).
 \end{equation}
Therefore using the estimates discussed in the previous paragraph we obtain that $\zetaTZm(p)$ and $\zeta(p)$
 both have the leading order asymptotic upper bound.  All of the estimates discussed in this sub-section together give the proof of Theorem \ref{main.thm}.

\subsection{Alternative representations of $\tilde{\zeta}(p)$}
In this section, we introduce two additional representations of $\tilde{\zeta}(p)$ from \eqref{tildezeta}.  In particular it is shown in Appendix \ref{sec:derivation} that we have the following splitting of $\tilde{\zeta}=\zeta_0+\zetaL$ in \eqref{zeta0B.appendix} with $c'>0$ as
\begin{multline}\label{zeta0B}
\zeta_0 \eqdef \frac{c'}{p^0}e^{\frac{p^0}{4}}\int_{\rth}\frac{dq}{q^0}\frac{e^{-\frac{3}{4}q^0}}{g}\int_{0}^\infty \frac{rdr}{\sqrt{r^2+s}}s_\Lambda\sigma(g_\Lambda,\theta_\Lambda) \frac{s\Phi(g)g^4}{s_\Lambda  \Phi(g_\Lambda)g^4_\Lambda}
\\
\times
\Big[\exp\left(-\frac{p^0+q^0}{4}\right) - \exp\left(-\frac{p^0+q^0}{4\sqrt{s}}\sqrt{r^2+s}\right)I_0\left(\frac{|p\times q|}{2g\sqrt{s}}r\right) \Big],
\end{multline}
and $\zetaL$ is further given in \eqref{zetaLB}.  Here $g_\Lambda$ is given by \eqref{g2}, $s_\Lambda$ by \eqref{slambda.def}, and $\theta_\Lambda$ by \eqref{coslam}.  Note that \eqref{zeta0B} and \eqref{zetaLB} can alternatively be obtained by  applying the change of variables $y \mapsto r = \sqrt{s} y$ to the expressions \eqref{zeta0By} and \eqref{zetaLBy}.   We will use the formula \eqref{zeta0B} in the proof of Proposition \ref{prop.zeta0.asymptotic}.

We can also write \eqref{zeta0B} in a further alternative form with other variables by using the following change of variables \begin{equation}\label{changeofv}r\mapsto k\eqdef \frac{1}{2}\sqrt{s}(\sqrt{r^2+s}-\sqrt{s}).\end{equation} 
Then this gives $$dk=\frac{1}{2}\sqrt{s}\frac{rdr}{\sqrt{r^2+s}}.$$
Also, we have 
$$ 
\frac{\sqrt{r^2+s}}{\sqrt{s}}=1+\frac{2k}{s},
$$ 
and 
$$
r=\frac{2\sqrt{k^2+ks}}{\sqrt{s}}.
$$
Here $g_\Lambda$ from \eqref{g2} and $\theta_\Lambda$ from \eqref{coslam} now take the form
\begin{equation}\label{gl}
g^2_\Lambda=g^2+k,
\end{equation}
and
$$\cos\theta_\Lambda =\frac{g^2-k}{g^2+k}=1-2\frac{k}{g^2+k}.$$
Therefore, we have 
\begin{equation}\label{tl}\frac{\theta_\Lambda}{2}\approx \sin\frac{\theta_\Lambda}{2}=\sqrt{\frac{k}{g^2+k}}.
\end{equation}
With respect to the new variable $k$, then \eqref{zeta0B} can be re-written as
\begin{multline}\notag
\zeta_0 \eqdef \frac{c'}{p^0}e^{\frac{p^0}{4}}\int_{\rth}\frac{dq}{q^0}\frac{e^{-\frac{3}{4}q^0}}{g}\int_{0}^\infty \frac{2dk}{\sqrt{s}}s_\Lambda\sigma(g_\Lambda,\theta_\Lambda)\frac{s\Phi(g)g^4}{s_\Lambda  \Phi(g_\Lambda)g^4_\Lambda}  
\\
\times
\left[\exp\left(-\frac{p^0+q^0}{4}\right) - \exp\left(-\frac{p^0+q^0}{4}\left(1+\frac{2k}{s}\right)\right)I_0\left(\frac{|p\times q|}{gs}\sqrt{k^2+ks}\right)  \right]\\
=\frac{c'}{p^0}\int_{\rth}\frac{dq}{q^0}\frac{e^{-q^0}}{g}\int_{0}^\infty \frac{2dk}{\sqrt{s}}s_\Lambda\sigma(g_\Lambda,\theta_\Lambda)\frac{s\Phi(g)g^4}{s_\Lambda  \Phi(g_\Lambda)g^4_\Lambda} \\
\times
\left[1- \exp\left(-\frac{(p^0+q^0)k}{2s}\right) I_0\left(\frac{|p\times q|}{gs}\sqrt{k^2+ks}\right) \right],
\end{multline}
where $s_\Lambda=g^2_\Lambda+4$ with \eqref{gl}.  This representation of $\zeta_0$ in the $k$ variables above will be used in the proof of Lemma \ref{lemma.zeta1}, which is one part of the leading order upper bound estimate of $\zeta_0$.

\subsection{Some useful pointwise estimates}

In this subsection we will state the following two lemma's which contain a collection of useful estimates that will be used throughout the remainder of this article.   Many of the estimates in Lemma \ref{lem:useful.ests} and Lemma \ref{lem:integral.ests} are from  \cite{GS3}.  We explain the full proofs of these estimates in Appendix \ref{sec:pointwiseProof} for the sake of completeness.

\begin{lemma}\label{lem:useful.ests}
With the notations \eqref{s} and \eqref{g} we have
\begin{equation}\label{s.ge.g2}
    s=g^2+4 \ge \max\{g^2,4\},
\end{equation}
and
\begin{equation}\label{s.le.pq}
    s\le 4p^0 q^0.
\end{equation}
We trivially conclude from \eqref{s.le.pq}-\eqref{s.ge.g2} that 
\begin{equation}\label{g.le.sqrtpq}
    g\lesssim \sqrt{p^0 q^0}
\end{equation}
Recalling \eqref{g.ineq.sharp} we further have
\begin{equation}\label{g.ge.lower}
\frac{|p-q|}{\sqrt{p^0q^0}}\leq g,
\end{equation}
and
\begin{equation}\label{g.ge.2lower}
\frac{|p\times q|}{\sqrt{p^0q^0}}\leq g,
\end{equation}
and
\begin{equation}\label{g.le.upper}
  g   \leq |p-q|.
\end{equation}
We also have
\begin{equation}\label{p0q0.le.pq}
  |p^0 - q^0|   \leq |p-q|,
\end{equation}
and
\begin{equation}\label{p0.plus.q0.le.p0q0}
  p^0 + q^0   \leq 2p^0 q^0.
\end{equation}

We now state a few pointwise estimates for \eqref{lj}.  
We have the inequality:
\begin{equation}\label{j.le.l}
    j\le l.
\end{equation}
Further
\begin{equation}\label{l.upper.ineq}
j^2   \leq \frac{1}{2}p^0 q^0, \quad l   \leq \frac{1}{2}p^0 q^0.
\end{equation}
Also
\begin{equation}\label{l2j2}
	l^2-j^2=\frac{(p^0+q^0)^2g^2-4|p\times q|^2}{16g^2}=\frac{s}{16g^2}|p-q|^2,
\end{equation} 
and
\begin{equation}\label{l2j2size}
	\sqrt{l^2-j^2}=|p-q|\frac{\sqrt{g^2+4}}{4g}\geq \frac{1}{4}|p-q|.
\end{equation}

Next for $g_\Lambda^2$ defined in \eqref{g2y.variable} we have
\begin{equation}\label{ineq.gL.here}
g^2 \max\{\sqrt{y^2+1}, \sqrt{2}\} g^2
\lesssim
 g_\Lambda^2 \lesssim   s_\Lambda \lesssim s \sqrt{y^2+1}, \quad \forall 0 \le y \le \infty.
\end{equation}
\end{lemma}

Further, in the proofs below we will also need sharp estimates of the following integrals
\begin{equation}\label{int.Kgamma}
    \bar{K}_\gamma(l,j)\eqdef \int_0^1dy\  y^{1-\gamma}  \exp(- l\sqrt{y^2+1})I_0(jy),
\end{equation}
where $\gamma \in (0,2)$
and
\begin{equation}\label{J2.special}
J_2(l,j)\eqdef \int_{0}^\infty\ dy\ \frac{y\exp\left(-l\sqrt{y^2+1}\right) I_0(j y)}{\sqrt{1+y^2}}
\end{equation} 
Also define
\begin{equation}\label{tildeK2def}
\tilde{K}_2(l,j)\eqdef \int_{0}^\infty\ dy\ y(y^2+1)^{1/2}\exp\left(-l\sqrt{y^2+1}\right) I_0(j y).
\end{equation} 

These integrals are known from \cite{Gradshteyn:1702455} and \cite{GS3}.  In particular a proof of Lemma \ref{lem:integral.ests} below is given by combining the results from \cite[Lemma 3.5, Lemma 3.6 and Corollary 2]{GS3}.   We give the following lemma and proof for completeness.

\begin{lemma}\label{lem:integral.ests}
For \eqref{lj}, we have
\begin{equation}\label{max.bound}\max_{0\leq x\leq 1}\exp(-l\sqrt{x^2+1}+jx) \lesssim \exp(-\sqrt{l^2-j^2}).\end{equation}
Then for \eqref{int.Kgamma}, we have the uniform estimate
	\begin{equation}\label{smally.lemma}
|\bar{K}_\gamma(l,j)|
	\lesssim
	 \exp\left(-\sqrt{l^2-j^2}\right).
	\end{equation}
For \eqref{J2.special} we have the exact formula
\begin{equation}\label{J2.lemma}
J_2(l,j)= (\sqrt{l^2-j^2})^{-1}\exp(-\sqrt{l^2-j^2}),
\end{equation} 
and then for \eqref{tildeK2def} we have the formula
\begin{multline}\label{k2lj.lemma}
\tilde{K}_2(l,j)= (\sqrt{l^2-j^2})^{-5}\exp(-\sqrt{l^2-j^2})\\\times \left((l^2-j^2+3\sqrt{l^2-j^2}+3)l^2-(l^2-j^2)-(\sqrt{l^2-j^2})^3\right).
\end{multline}	
\end{lemma}

Lemma's \ref{lem:useful.ests} and \ref{lem:integral.ests} will be proven in Appendix \ref{sec:pointwiseProof}.

\subsection{Related literature}  In this section, we briefly give citations for a small number of closely related results on the mathematical theory on the Boltzmann equation. Regarding the Newtonian Boltzmann equation without angular cutoff, we mention the works  \cite{MR839310, MR4049224, MR3551261, MR3572500, MR4107942, MR3665667, MR2847536, MR2679369, MR2793203, MR2629879, MR2784329}. 
The mathematical theory on the relativistic Boltzmann equation with angular cutoff includes \cite{GS3,  MR1275397, MR2728733, MR2891870, MR1402446, MR2102321, Cal, MR2679588, Dudynski2, D-E3, D-E2, MR818441,wang2018global,duan2017relativistic,lee2013spatially,ha2009uniform,ha2007asymptotic}. 
This list includes in particular the discussions on the global well-posedness and stability, the regularity of solutions, the blow-up with the gain-term only, the Newtonian limit, and the causality. Further detailed discussions of the related literature are found in our companion paper \cite{RelBolNoncut2020}.

\subsection{Outline} \label{outline}
To prove our main Theorem \ref{main.thm}, we will use the two different representations that we have given in Section \ref{sec:main.decomp} and these will be derived in Appendix \ref{sec:derivation}.  We will further follow the proof strategy that we have already outlined in Section \ref{sec:proof.outline}.  We  prove that $\zeta$ from \eqref{def.zeta} has a leading order positive lower bound in Proposition \ref{prop.coercive} in Section \ref{sec:leadingorder lower bound zeta}.  Then we will prove that $\zeta_0$  from \eqref{zeta0By} has the leading order upper bound in Proposition \ref{prop.zeta0.asymptotic} in Section \ref{sec:fullsharpupper zeta0}.    In Section \ref{sec:low order zeta1}, we prove in Proposition \ref{prop.zetaL.asymptotic}  that $\zetaL(p)$ from \eqref{zetaLBy} has a lower order upper bound and we further prove in Proposition \ref{prop.zetaL.asymptoticnew} that $\zetaTL(p)$ from \eqref{eq:tildezetanew1L} has a lower order upper bound.

\section{Leading order lower bound estimate}\label{sec:leadingorder lower bound zeta}

The main result in this section is the following leading order lower bound.

\begin{proposition}\label{prop.coercive}	Suppose $\gamma \in (0,2)$ in \eqref{angassumption}. Then for both hard \eqref{hard} and soft \eqref{soft} interactions, using the notation \eqref{rho.def}, for \eqref{def.zeta}, we have
	$$\zeta(p) \gtrsim (p^0)^{\frac{\singS+\gamma}{2}}.$$
This uniform lower bound also holds for \eqref{def.zeta.mod}.
\end{proposition}

\begin{proof}[Proof of Proposition \ref{prop.coercive}]In order to obtain the lower-bound estimate for $\zeta(p)$, we first study the lower bound of the perfect square term $$\Big[\exp(l-l\sqrt{y^2+1} +jy\cos\phi)-1\Big]^2$$ in \eqref{def.zeta}. We first observe that, if $y\in [0,y^*]$ with
\begin{equation}\label{def.y.star}
y^*\eqdef \frac{2lj\cos\phi}{l^2-j^2\cos^2\phi},
\end{equation}
then we have$$ l-l\sqrt{y^2+1} +jy\cos\phi\ge 0.$$
Notice that we also have
\begin{equation}\notag
    l-l\sqrt{y^2+1}+jy\cos\phi \ge \frac{1}{2} jy\cos\phi, \quad
\end{equation}
if  $0 \le y \le y_1,$ where 
\begin{equation}\notag
y_1 \eqdef \frac{2l j \cos\phi }{4l^2-j^2 \cos^2\phi}.
\end{equation}
Also $\frac{1}{2} \le \cos\phi \le \frac{\sqrt{2}}{2}$ for $\phi\in [\pi/4,\pi/3]$.
Recalling \eqref{def.y.star}, we remark that $y^*\le 3$ because
$$
y^* = \frac{2lj\cos\phi}{l^2-j^2\cos^2\phi}\le \frac{
\sqrt{2}lj}{l^2-\frac{j^2}{2}}\le \frac{
\sqrt{2}l^2}{\frac{l^2}{2}}\le 2 \sqrt{2},
$$ as $j\le l$ and $\frac{
\sqrt{2}lj}{l^2-\frac{j^2}{2}}$ is an increasing function in $j$.
Recalling again \eqref{def.y.star}, then $0\le y_1\le \frac{y^*}{4}\le \frac{\sqrt{2}}{2}$.
Since in particular with $\phi\in [\pi/4,\pi/3]$ we have
\begin{equation}\notag
    l-l\sqrt{y^2+1}+jy\cos\phi \ge 0, \quad 0 \le y \le y_1.
\end{equation}
Then by the Taylor expansion
\begin{equation}\notag
    \Big[\exp(l-l\sqrt{y^2+1} +jy\cos\phi)-1\Big]^2
    \ge (l-l\sqrt{y^2+1}+jy\cos\phi)^2.
\end{equation}
We will use this lower bound in the following developments.

Now we start by proving the stated lower bound for \eqref{def.zeta.mod}.
Now we split each integral representation of the decomposed pieces based on a restriction of the $y$ and $\phi$ domains.  We will now define the term 
$$
\zeta_* \eqdef [\zeta']_{0\le y\le y_1\text{ and }\phi\in [\pi/4,\pi/3]},
$$
where $\zeta_*$ is $\zeta'$ when the integrals inside \eqref{def.zeta.mod} are only on the restricted domains $0\le y\le y_1\text{ and }\phi\in [\pi/4,\pi/3]$. 
This notation is similar to \eqref{convention}.   Note that of course $\zeta'(p) \ge \zeta_*(p)$.  We will show that $\zeta_*(p)$ has a high-order lower bound. Note that inside this integration region, $0\le y\le y_1\text{ and }\phi\in [\pi/4,\pi/3]$, the integral is still non-negative.

First of all, we note from \eqref{def.zeta.mod} that
\begin{multline*}
\zeta_*(p)
\ge \frac{c'}{2\pi p^0}\int_{\rth}\frac{dq}{q^0}\frac{e^{-q^0}\sqrt{s}}{g}\int_{\pi/4}^{\pi/3} d\phi \int^{y_1}_{0}\frac{ydy}{\sqrt{y^2+1}}s_\Lambda\sigma(g_\Lambda,\theta_\Lambda)
\\
\times (l-l\sqrt{y^2+1}+jy\cos\phi)^2\frac{s\Phi(g)g^4}{s_\Lambda  \Phi(g_\Lambda)g^4_\Lambda}+\langle p \rangle^{(\rho+\gamma)/2} 1_{|p|\le 1}
\\
\ge
\frac{c'}{2\pi p^0}\int_{\rth}\frac{dq}{q^0}\frac{e^{-q^0}\sqrt{s}}{g}\int_{\pi/4}^{\pi/3} d\phi \int^{y_1}_{0}\frac{ydy}{\sqrt{y^2+1}}s_\Lambda\sigma(g_\Lambda,\theta_\Lambda)\\\times\frac{1}{4}\left( jy\cos\phi \right)^2\frac{s\Phi(g)g^4}{s_\Lambda  \Phi(g_\Lambda)g^4_\Lambda}+\langle p \rangle^{(\rho+\gamma)/2} 1_{|p|\le 1}
\\
\ge \frac{c'}{2\pi p^0}\int_{\rth}\frac{dq}{q^0}\frac{e^{-q^0}\sqrt{s}}{g}\int_{\pi/4}^{\pi/3} d\phi \int^{y_1}_{0}\frac{ydy}{\sqrt{y^2+1}}s_\Lambda\sigma(g_\Lambda,\theta_\Lambda)\\\times
\frac{1}{4}\left(\frac{jy}{2}\right)^2\frac{s\Phi(g)g^4}{s_\Lambda  \Phi(g_\Lambda)g^4_\Lambda}+\langle p \rangle^{(\rho+\gamma)/2} 1_{|p|\le 1},
\end{multline*} where we used that $\cos\phi\ge \frac{1}{2}$ when $\phi\in[\pi/4,\pi/3].$

 Now we will estimate the kernel $\sigma(g_\Lambda,\theta_\Lambda)$ from \eqref{kernel.product}.  Here, by \eqref{angassumption} with \eqref{g2y.variable}, \eqref{changeofv}, \eqref{gl} and \eqref{tl} we have
\begin{equation}\label{ang.equiv}
\sigma_0(\theta_\Lambda)\approx  \left(\frac{sy^2}{sy^2+2g^2(\sqrt{y^2+1}+1)}\right)^{-1-\gamma/2}.
\end{equation}
Next using \eqref{g2y.variable} we have that 
\begin{multline}\label{glambda.calc}
g^2_\Lambda=g^2+\frac{s}{2}(\sqrt{y^2+1}-1)=
g^2+\frac{sy^2}{2(\sqrt{y^2+1}+1)}\\
=\frac{sy^2+2g^2(\sqrt{y^2+1}+1)}{2(\sqrt{y^2+1}+1)}.
\end{multline}
 Thus, also recalling \eqref{rho.def} and \eqref{glambda.calc}, we have
\begin{multline}\notag
s_\Lambda\sigma(g_\Lambda,\theta_\Lambda)\frac{s\Phi(g)g^4}{s_\Lambda  \Phi(g_\Lambda)g^4_\Lambda}=s\Phi(g)	\sigma_0(\theta_\Lambda)\frac{g^4}{g^4_\Lambda}\\
\approx s\Phi(g)g^4 \left(\frac{sy^2}{sy^2+2g^2(\sqrt{y^2+1}+1)}\right)^{-1-\gamma/2}\frac{1}{g^4_\Lambda}.
\end{multline}
Thus
\begin{multline}\notag
s_\Lambda\sigma(g_\Lambda,\theta_\Lambda)\frac{s\Phi(g)g^4}{s_\Lambda  \Phi(g_\Lambda)g^4_\Lambda}
\approx sg^{\singS+4} \left(\frac{sy^2}{sy^2+2g^2(\sqrt{y^2+1}+1)}\right)^{-1-\gamma/2}
\\
\times\left(\frac{sy^2+2g^2(\sqrt{y^2+1}+1)}{2(\sqrt{y^2+1}+1)}\right)^{-2}.
\end{multline}
We conclude that
\begin{multline}\label{kernel.equiv}
s_\Lambda\sigma(g_\Lambda,\theta_\Lambda)\frac{s\Phi(g)g^4}{s_\Lambda  \Phi(g_\Lambda)g^4_\Lambda}=s\Phi(g)	\sigma_0(\theta_\Lambda)\frac{g^4}{g^4_\Lambda}\\
\approx s^{-\gamma/2}g^{\singS+4}y^{-2-\gamma}(2(\sqrt{y^2+1}+1))^{2}\left(sy^2+2g^2(\sqrt{y^2+1}+1)\right)^{-1+\gamma/2}.
\end{multline}
Thus, since $\gamma \in (0,2)$, we have 
\begin{equation*}
 s_\Lambda\sigma(g_\Lambda,\theta_\Lambda)\frac{s\Phi(g)g^4}{s_\Lambda  \Phi(g_\Lambda)g^4_\Lambda}   \gtrsim s^{-1}y^{-2-\gamma} g^{\rho+4},
\end{equation*}
where we used $\gamma/2-1<0$ and $y\le y_1\le  \frac{\sqrt{2}}{2}$. 
Therefore,
\begin{multline*}
\zeta_*(p)
\gtrsim  \frac{c'}{ p^0}\int_{\rth}\frac{dq}{q^0}\frac{e^{-q^0}\sqrt{s}}{g} \int_{\pi/4}^{\pi/3} d\phi\\\times  \int^{y_1}_{0}\frac{ydy}{\sqrt{y^2+1}}j^2s^{-1}y^{-\gamma} g^{4+\rho}+\langle p \rangle^{(\rho+\gamma)/2} 1_{|p|\le 1}.
\end{multline*} 
Then we have
\begin{multline*}
\zeta_*(p)
\gtrsim  \frac{c'}{ p^0}\int_{\rth}\frac{dq}{q^0}\frac{e^{-q^0}\sqrt{s}}{g} 
j^2s^{-1} g^{4+\rho}
\int^{y_1}_{0} y^{1-\gamma}dy+\langle p \rangle^{(\rho+\gamma)/2} 1_{|p|\le 1}
\\
\gtrsim  \frac{c'}{ p^0}\int_{\rth}\frac{dq}{q^0}\frac{e^{-q^0}\sqrt{s}}{g} 
j^2s^{-1} g^{4+\rho}
 y_1^{2-\gamma}+\langle p \rangle^{(\rho+\gamma)/2} 1_{|p|\le 1}.
\end{multline*} 
We further have on $\phi \in [\pi/4, \pi/3]$, using also $j \le l$, that
\begin{equation}\notag
y_1^{2-\gamma}
= \left(\frac{2lj\cos\phi}{4l^2-j^2\cos^2\phi}\right)^{2-\gamma}
\geq \left(\frac{lj}{4l^2-j^2/4}\right)^{2-\gamma}
\gtrsim
\left(\frac{j}{l}\right)^{2-\gamma},
\end{equation} 
as $\cos\phi \ge \frac{1}{2}$.  Altogether, we have
\begin{multline*}
\zeta_*(p)
\gtrsim  \frac{c'}{ p^0}\int_{\rth}\frac{dq}{q^0}\frac{e^{-q^0}\sqrt{s}}{g}  j^2s^{-1} g^{4+\rho} \left(\frac{j}{l}\right)^{2-\gamma}+\langle p \rangle^{(\rho+\gamma)/2} 1_{|p|\le 1}
\\
\gtrsim  \frac{c'}{ p^0}\int_{\rth}\frac{dq}{q^0}\frac{e^{-q^0}\sqrt{s}}{g}  s^{-1} g^{4+\rho} j^{4-\gamma}l^{-2+\gamma}+\langle p \rangle^{(\rho+\gamma)/2} 1_{|p|\le 1}. 
\end{multline*}
Now we recall \eqref{lj}, \eqref{s.ge.g2}, \eqref{g.ge.lower} and note that $\gamma\in(0,2)$. 
Then we obtain
\begin{multline}\label{eq.sameuntilhere}
\zeta_*(p)
\gtrsim   \frac{1}{ p^0}\int_{\rth}\frac{dq}{q^0}e^{-q^0}  g^{-1+\gamma+\rho} |p\times q|^{4-\gamma}s^{-1/2}(p^0+q^0)^{-2+\gamma} +\langle p \rangle^{(\rho+\gamma)/2} 1_{|p|\le 1}\\
\gtrsim   \frac{1}{ p^0}\int_{\rth}\frac{dq}{q^0}e^{-q^0}  s^{-1/2}g^{\gamma+\rho-1} |p\times q|^{4-\gamma}(p^0q^0)^{-2+\gamma}+\langle p \rangle^{(\rho+\gamma)/2} 1_{|p|\le 1},
\end{multline}
above we also used \eqref{p0.plus.q0.le.p0q0}.
Further, since \eqref{s.le.pq}, we have 
$$s^{-1/2}\gtrsim (p^0q^0)^{-1/2}.$$
If $\gamma+\rho -1\ge 0$, 
then from \eqref{g.ge.lower} and \eqref{p0q0.le.pq} we have
$$g^{\gamma+\rho-1}\ge \left(\frac{|p-q|}{\sqrt{p^0q^0}}\right)^{\gamma+\rho-1}\ge \left(\frac{|p^0-q^0|}{\sqrt{p^0q^0}}\right)^{\gamma+\rho-1}.$$ Otherwise, when $\gamma+\rho -1< 0$, using \eqref{s.le.pq} with \eqref{s.ge.g2} we have
$$g^{\gamma+\rho-1}\ge (p^0q^0)^{\gamma/2+\rho/2-1/2}.$$
Finally, we use the spherical-coordinate representation of $q\mapsto (r,\theta_q,\phi_q)$. We let the $z$-axis be parallel to the direction of $p$ such that $\phi_q$ is the angle between $p$ and $q$. Then we have
\begin{multline}\label{eq.q.last}
\zeta_*(p)
\gtrsim \frac{1}{p^0}\int_0^{\infty}dr \frac{r^2e^{-\sqrt{1+r^2}}}{\sqrt{r^2+1}} \int_0^\pi d\phi_q \ \sin\phi_q   \\\times (p^0q^0)^{-1/2}\min\left\{\left(\frac{|p^0-q^0|}{\sqrt{p^0q^0}}\right)^{\gamma+\rho-1}, (p^0q^0)^{\gamma/2+\rho/2-1/2}\right\}\\\times |p|^{4-\gamma}r^{4-\gamma} \sin^{4-\gamma}\phi_q(p^0q^0)^{-2+\gamma}+\langle p \rangle^{(\rho+\gamma)/2} 1_{|p|\le 1}\\
\approx |p|^{4-\gamma} (p^0)^{-1-1/2+\gamma/2+\rho/2-1/2-2+\gamma}+\langle p \rangle^{(\rho+\gamma)/2} 1_{|p|\le 1}\\
\approx |p|^{4-\gamma}(p^0)^{-4+3\gamma/2+\rho/2} +\langle p \rangle^{(\rho+\gamma)/2} 1_{|p|\le 1}.
\end{multline}
Now we remark that if $|p|\ge 1$ then we have $|p|\approx p^0$. We conclude
$$
\zeta'(p) 
\ge 
\zeta_*(p)\gtrsim (p^0)^{\frac{\rho}{2}+\frac{\gamma}{2}}.$$
This completes the proof for the high-order lower bound of $\zeta'(p)$.

Similarly, we can obtain the high-order lower bound of $\zeta(p)$ from \eqref{def.zeta}. Note that the only difference between $\zeta(p)$ and $\zeta'(p)$ from \eqref{def.zeta} and \eqref{def.zeta.mod} is that the domain $\rth$ with respect to $q$ variable in \eqref{def.zeta.mod} is now restricted to $|q|\le \frac{1}{2}|p|^{1/m}$ in \eqref{def.zeta}. Then we note that the proof for the high-order lower bound of $\zeta(p)$ is exactly the same as $\zeta'(p)$ until 
\eqref{eq.sameuntilhere} above except for the change from $\int_\rth dq$ into $\int_{|q|\le\frac{1}{2}|p|^{1/m}} dq $. 
Then in the spherical-coordinate representation of $q\mapsto (r,\theta_q,\phi_q)$ for \eqref{eq.q.last}, we change the integral domain $\int_0^\infty dr$ in \eqref{eq.q.last} to $\int_0^{\frac{1}{2}|p|^{1/m}} dr$. Then analogous to \eqref{eq.q.last} we have
\begin{multline}\notag
\zeta(p)
\gtrsim 
\frac{1}{p^0}\int_0^{\frac{1}{2}|p|^{1/m}} dr \frac{r^2e^{-\sqrt{1+r^2}}}{\sqrt{r^2+1}} \int_0^\pi d\phi_q \ \sin\phi_q   \\\times (p^0q^0)^{-1/2}\min\left\{\left(\frac{|p^0-q^0|}{\sqrt{p^0q^0}}\right)^{\gamma+\rho-1}, (p^0q^0)^{\gamma/2+\rho/2-1/2}\right\}\\\times |p|^{4-\gamma}r^{4-\gamma} \sin^{4-\gamma}\phi_q(p^0q^0)^{-2+\gamma}+\langle p \rangle^{(\rho+\gamma)/2} 1_{|p|\le 1}.
\end{multline}
Now in the region ${|q|\le\frac{1}{2}|p|^{1/m}}$  with $|p| \ge 1$ and $m\ge 1$ sufficiently large inside \eqref{eq.q.last}  we have 
\begin{multline*}
    \min\left\{\left(\frac{|p^0-q^0|}{\sqrt{p^0q^0}}\right)^{\gamma+\rho-1}, (p^0q^0)^{\gamma/2+\rho/2-1/2}\right\}
    \\
    \gtrsim
    (p^0)^{\gamma/2+\rho/2-1/2}
        \min\left\{(q^0)^{-\gamma/2-\rho/2+1/2}, (q^0)^{\gamma/2+\rho/2-1/2}\right\}
\end{multline*}
We further have on $|p| \ge 1$ with $q^0 = \sqrt{1+r^2}$ that 
\begin{multline}\notag
\int_0^{\frac{1}{2}|p|^{1/m}} dr  \frac{r^2e^{-\sqrt{1+r^2}}}{\sqrt{r^2+1}} 
r^{4-\gamma}
\min\left\{(q^0)^{-\gamma/2-\rho/2}, (q^0)^{\gamma/2+\rho/2-1}\right\} (q^0)^{-2+\gamma}
\\
\gtrsim
\int_0^{\frac{1}{2}} dr  \frac{r^2e^{-\sqrt{1+r^2}}}{\sqrt{r^2+1}} 
r^{4-\gamma}
\min\left\{(q^0)^{-\gamma/2-\rho/2}, (q^0)^{\gamma/2+\rho/2-1}\right\} (q^0)^{-2+\gamma}
\gtrsim
c_{\frac{1}{2}},
\end{multline}
for some constant $c_{\frac{1}{2}}>0$ if $|p|\ge 1.$ 
Therefore, the same proof with the modifications above works for the leading-order lower bound of $\zeta(p)$ from \eqref{def.zeta}.  In particular the estimate \eqref{eq.q.last} continues to hold,  and this completes the leading-order lower-bound estimates.   
\end{proof}

This completes the leading order lower bound estimates of $\zeta(p)$. In the next two sections, we will use the decomposition $\tilde{\zeta}(p)=\zeta_0(p)+\zeta_L(p)$ from \eqref{zeta0By} and \eqref{zetaLBy} to obtain the leading order upper bound of $\tilde{\zeta}(p)$, and the lower order upper bounds of $\zetaL(p)$ and $\zetaTL(p)$ from \eqref{eq:tildezetanew1L}.

\section{Leading order upper bound estimates}\label{sec:fullsharpupper zeta0}

We now prove the following leading order upper bound estimate for $\zetaZ$ from \eqref{zeta0By} using the alternative representation \eqref{zeta0B}:

\begin{proposition}\label{prop.zeta0.asymptotic}
	Suppose $\gamma \in (0,2)$  in \eqref{angassumption}. Then for both hard \eqref{hard} and soft \eqref{soft} interactions, for \eqref{zeta0B} when $|p|\ge 1$, we have
	$$|\zetaZ(p)|\lesssim (p^0)^{\frac{\singS+\gamma}{2}}.$$
This consequently implies the same uniform bound for $\zetaZm(p)$ from \eqref{def.convention.zeta0By}.	
\end{proposition}

For the proof, 
we decompose $\zetaZ$ from \eqref{zeta0B} as $\zetaZ=\zeta_1+\zeta_2$ where
\begin{multline}\label{zeta1B}
\zeta_1 \eqdef \frac{c'}{p^0}e^{\frac{p^0}{4}}\int_{\rth}\frac{dq}{q^0}\frac{e^{-\frac{3}{4}q^0}}{g}\int_{0}^\infty \frac{rdr}{\sqrt{r^2+s}}s_\Lambda\sigma(g_\Lambda,\theta_\Lambda)  \frac{s\Phi(g)g^4}{s_\Lambda  \Phi(g_\Lambda)g^4_\Lambda}
\\
\times
\left[\exp\left(-\frac{p^0+q^0}{4}\right) - \exp\left(-\frac{p^0+q^0}{4\sqrt{s}}\sqrt{r^2+s}\right) \right],
\end{multline}
\begin{multline}\label{zeta2B}
\zeta_2 \eqdef \frac{c'}{p^0}e^{\frac{p^0}{4}}\int_{\rth}\frac{dq}{q^0}\frac{e^{-\frac{3}{4}q^0}}{g}\int_{0}^\infty \frac{rdr}{\sqrt{r^2+s}}s_\Lambda\sigma(g_\Lambda,\theta_\Lambda)  \frac{s\Phi(g)g^4}{s_\Lambda  \Phi(g_\Lambda)g^4_\Lambda}
\\
\times
\exp\left(-\frac{p^0+q^0}{4\sqrt{s}}\sqrt{r^2+s}\right)\left[1-I_0\left(\frac{|p\times q|}{2g\sqrt{s}}r\right) \right].
\end{multline}
Clearly, $\zeta_1$ is positive.   We estimate $\zeta_1$ in Lemma \ref{lemma.zeta1} and then we will estimate $\zeta_2$ in Lemma \ref{lemma.zeta2};  Proposition \ref{prop.zeta0.asymptotic} then follows directly.  First, we have

\begin{lemma}\label{lemma.zeta1}
	Assuming either \eqref{hard} or \eqref{soft} with \eqref{angassumption}, then we have the following uniform asymptotic bound for $\zeta_1$ from \eqref{zeta1B}:
	$$\zeta_1(p)\lesssim (p^0)^{\frac{\singS+\gamma}{2}}.$$
\end{lemma}

\begin{proof}The change of variables \eqref{changeofv} on the representation \eqref{zeta1B} yields that
\begin{multline*}\zeta_1 \eqdef \frac{c'}{p^0}e^{\frac{p^0}{4}}\int_{\rth}\frac{dq}{q^0}\frac{e^{-\frac{3}{4}q^0}}{g}\int_{0}^\infty \frac{2dk}{\sqrt{s}}s_\Lambda\sigma(g_\Lambda,\theta_\Lambda)  \frac{s\Phi(g)g^4}{s_\Lambda  \Phi(g_\Lambda)g^4_\Lambda}
\\
\times
\Big[\exp\left(-\frac{p^0+q^0}{4}\right) - \exp\left(-\frac{p^0+q^0}{4}\left(1+\frac{2k}{s}\right)\right) \Big]\\
=\frac{c'}{p^0}\int_{\rth}\frac{dq}{q^0}\frac{e^{-q^0}}{g}\int_{0}^\infty \frac{2dk}{\sqrt{s}}s_\Lambda\sigma(g_\Lambda,\theta_\Lambda)  \frac{s\Phi(g)g^4}{s_\Lambda  \Phi(g_\Lambda)g^4_\Lambda}
\Big[1- \exp\left(-\frac{(p^0+q^0)k}{2s}\right) \Big]\\
=\frac{c_4}{p^0}\int_{\rth}\frac{dq}{q^0}e^{-q^0}\sqrt{s}\Phi(g)g^3\int_{0}^\infty dk \frac{\sigma_0(\cos\theta_{\Lambda})}{ g^4_\Lambda}
\Big[1- \exp\left(-\frac{(p^0+q^0)k}{2s}\right) \Big].
\end{multline*}
Here $c_4 = 2c'$.
	We start by showing the upper-bound estimates of $\zeta_1$. By the fundamental theorem of calculus, we have
	\begin{multline}\label{zetanew}
	\zeta_1=\frac{c_4}{p^0}\int_{\rth}\frac{dq}{q^0}e^{-q^0}\sqrt{s}\Phi(g)g^3\int_{0}^\infty dk ~\frac{\sigma_0(\cos\theta_{\Lambda})}{ g^4_\Lambda}
	\Big[1- \exp\left(-\frac{(p^0+q^0)k}{2s}\right) \Big]\\
	=\frac{c_4}{p^0}\int_{\rth}\frac{dq}{q^0}e^{-q^0}\sqrt{s}\Phi(g)g^3\int_{0}^\infty dk ~\frac{\sigma_0(\cos\theta_{\Lambda})}{ g^4_\Lambda}
	\\\times	\int_0^1d\vartheta\ \exp\left(-\frac{(p^0+q^0)k}{2s}\vartheta\right)\frac{(p^0+q^0)k}{2s}.
	\end{multline}
	Note that using \eqref{angassumption}, \eqref{hard}, \eqref{soft},  \eqref{rho.def}, \eqref{gl}, and \eqref{tl}, we have
\begin{equation} \label{new.ang.est.zeta.doit}
\Phi(g)\approx  g^{\singS}
\ \text{and} \
\sigma_0(\cos\theta_{\Lambda})\approx\left(\frac{k}{k+g^2}\right)^{-1-\gamma/2}\approx g_\Lambda^{2+\gamma} k^{-1-\gamma/2}.
\end{equation}
We will use this equivalence in the following developments.

We split into two cases: $k\leq 4$ and $k>4$.  First consider $k\leq 4$.
We use $\exp\left(-\frac{(p^0+q^0)k}{2s}\vartheta\right)\leq 1$ and $g\leq g_\Lambda = (g^2 + k)^{1/2}$ from \eqref{gl}, then when $k\leq 4$ we have
	\begin{multline}\label{zeta.use.hard.too}
	\zeta_1(p)\lesssim \frac{1}{p^0}\int_{\rth}\frac{dq}{q^0}e^{-q^0}\sqrt{s}g^{3}\Phi(g)\int_{0}^{4} dk \frac{g_\Lambda^{2+\gamma}}{ g^4_\Lambda} k^{-1-\frac{\gamma}{2}}
	\frac{(p^0+q^0)k}{2s}
	\\
	\lesssim \int_{\rth}dq\ e^{-q^0}\Phi(g)s^{\frac{\gamma}{2}}\int_{0}^{4} dk \ k^{-\frac{\gamma}{2}}.
	\end{multline}
	Here we used $g_\Lambda \le \sqrt{s}$ when $k\leq 4$.
	Since $\gamma\in(0,2)$, the integral converges.
	
	Now we use $g\lesssim \sqrt{p^0q^0}$ from \eqref{s.ge.g2} and \eqref{s.le.pq} in the hard interaction \eqref{hard} case.  Alternatively we will use 
	$g\geq \frac{|p-q|}{\sqrt{p^0q^0}}$ from \eqref{g.ge.lower} in the soft interaction \eqref{soft} case, and $s\lesssim p^0q^0$  from \eqref{s.le.pq}.  Then on $k\le 4$ we further have
	\begin{equation}\notag
	\zeta_1(p)\lesssim \int_{\rth}dq\ e^{-q^0}(p^0q^0)^{\frac{\singA+\gamma}{2}}
	\lesssim
	(p^0)^{\frac{\singA+\gamma}{2}},
	\end{equation}for the hard interactions, and
	\begin{equation}\notag
	\zeta_1(p)\lesssim \int_{\rth}dq\ e^{-q^0}\left(\frac{|p-q|}{\sqrt{p^0q^0}}\right)^{-\singB}(p^0q^0)^{\frac{\gamma}{2}}
	\lesssim
	(p^0)^{\frac{-\singB+\gamma}{2}},
	\end{equation}
	for the soft interactions.

	On the other hand, when $k>4$, we still have \eqref{new.ang.est.zeta.doit} and \eqref{zetanew}.
Hence
\begin{multline}\label{k.ge.four.est.zeta}
	\zeta_1(p)
	\lesssim
	\frac{1}{p^0}\int_{\rth}\frac{dq}{q^0}e^{-q^0}\sqrt{s}g^{3}\Phi(g)\int_{4}^\infty dk \frac{\Big[1- \exp\left(-\frac{(p^0+q^0)k}{2s}\right) \Big]}{ k^{1+\gamma/2}(k+g^2)^{1-\gamma/2}}
\\
	\lesssim
	\frac{1}{p^0}\int_{\rth}\frac{dq}{q^0}e^{-q^0}\sqrt{s}g^{1+\gamma}\Phi(g)\int_{4}^\infty dk ~ k^{-1-\gamma/2}  \Big[1- \exp\left(-\frac{(p^0+q^0)k}{2s}\right) \Big]
\\
\lesssim
	\int_{\rth} dq ~ e^{-q^0}g^{\gamma}\Phi(g)\int_{4}^\infty dk ~ k^{-1-\gamma/2}
	\Big[1- \exp\left(-\frac{(p^0+q^0)k}{2s}\right) \Big]
\\
\lesssim
	\int_{\rth} dq ~ e^{-q^0}g^{\gamma}\Phi(g)\int_{4}^\infty dk ~ k^{-1-\gamma/2}.
\end{multline}
Above we used $g \lesssim \sqrt{s} \lesssim \sqrt{\pZ \qZ}$ from \eqref{s.ge.g2}-\eqref{s.le.pq} and $\Big[1- \exp\left(-\frac{(p^0+q^0)k}{2s}\right) \Big] \lesssim 1$.

 Then, also using $g\lesssim \sqrt{p^0q^0}$ for hard interactions \eqref{hard} and \eqref{g.ge.lower} for the soft interactions \eqref{soft}, when $k\geq 4$, we have
	\begin{multline*}
	\zeta_1(p)
	\lesssim
	\int_{\rth} dq ~ e^{-q^0}g^{\singA+\gamma}\int_{4}^\infty dk ~ k^{-1-\gamma/2}
		\lesssim
	\int_{\rth} dq ~ e^{-q^0}g^{\singA+\gamma}
\\
\lesssim
	\int_{\rth} dq ~ e^{-q^0}(p^0q^0)^{\frac{\singA+\gamma}{2}}
	\lesssim
	(p^0)^{\frac{\singA+\gamma}{2}},
	\end{multline*}in the hard interaction case, and
	\begin{multline*}
	\zeta_1(p)
	\lesssim
	\int_{\rth} dq ~ e^{-q^0}g^{-\singB+\gamma}\int_{4}^\infty dk ~ k^{-1-\gamma/2}
		\lesssim
	\int_{\rth} dq ~ e^{-q^0}g^{-\singB+\gamma}
\\
\lesssim
	\int_{\rth} dq ~ e^{-q^0}\left(\frac{|p-q|}{\sqrt{p^0q^0}}\right)^{-\singB+\gamma}
	\lesssim
	(p^0)^{\frac{-\singB+\gamma}{2}},
	\end{multline*}in the soft interaction case.
	This completes the upper-bound estimate of $\zeta_1$.  
\end{proof}

On the other hand, we have the following upper-bound estimate for $\zeta_2$:
\begin{lemma}\label{lemma.zeta2}
	Suppose $\gamma \in (0,2)$ in \eqref{angassumption}.  Then for both hard \eqref{hard} and soft \eqref{soft} interactions with \eqref{rho.def} we have the following uniform upper bound for \eqref{zeta2B} when $|p|\ge 1$:
	$$|\zeta_2(p)|\lesssim (p^0)^{\frac{\singS+\gamma}{2}}.$$
	\end{lemma}

	\begin{proof}
	We use the change of variables $r\mapsto y\eqdef \frac{r}{\sqrt{s}}$ on \eqref{zeta2B}. This yields
	\begin{multline}\notag
\zeta_2 \eqdef \frac{c'}{p^0}e^{\frac{p^0}{4}}\int_{\rth}\frac{dq}{q^0}\frac{e^{-\frac{3}{4}q^0}}{g}\int_{0}^\infty \frac{\sqrt{s}ydy}{\sqrt{y^2+1}}s_\Lambda\sigma(g_\Lambda,\theta_\Lambda)  \frac{s\Phi(g)g^4}{s_\Lambda  \Phi(g_\Lambda)g^4_\Lambda}
\\
\times
\exp\left(-\frac{p^0+q^0}{4}\sqrt{y^2+1}\right)\left[1-I_0\left(\frac{|p\times q|}{2g}y\right) \right].
\end{multline}
Recall \eqref{lj}.   Note that $\sigma(g_\Lambda,\theta_\Lambda)=\Phi(g_\Lambda)\sigma_0(\theta_\Lambda)\ge 0.$ Then we have
 \begin{multline}\label{zeta2New}
 \zeta_2 \eqdef \frac{c'}{p^0}e^{\frac{p^0}{4}}\int_{\rth}\frac{dq}{q^0}\frac{e^{-\frac{3}{4}q^0}}{g}\int_{0}^\infty \frac{\sqrt{s}ydy}{\sqrt{y^2+1}}s\Phi(g)\sigma_0(\theta_\Lambda)  \frac{g^4}{g^4_\Lambda}
 \\
 \times
 \exp\left(-l\sqrt{y^2+1}\right)\left[1-I_0(jy) \right].
 \end{multline} 
 Note that $I_0 \ge 1$ so that $\zeta_2 \le 0$.
 By \eqref{ang.equiv}, using $g^2\leq s$ from \eqref{s.ge.g2} we have
$$
	\sigma_0(\theta_\Lambda)
	\lesssim (y^2+\sqrt{y^2+1})^{1+\gamma/2}y^{-2-\gamma},
$$
Plugging this into \eqref{zeta2New}, we have
	 \begin{multline}\label{zeta2abs}
	|\zeta_2|\lesssim  \frac{1}{p^0}e^{\frac{p^0}{4}}\int_{\rth}\frac{dq}{q^0}\frac{e^{-\frac{3}{4}q^0}}{g}s\sqrt{s}\Phi(g)\int_{0}^\infty \frac{y^{-1-\gamma}dy}{\sqrt{y^2+1}}(y^2+\sqrt{y^2+1})^{1+\gamma/2}
	\\
	\times
	\exp\left(-l\sqrt{y^2+1}\right)\left[I_0(jy)-1 \right],
	\end{multline} where we also used $\frac{g^4}{g^4_\Lambda}\leq 1$ from \eqref{g2y.variable}. Also note that
	\begin{multline*}	\exp\left(-l\sqrt{y^2+1}\right)=\exp(-l)	\exp\left(-l(\sqrt{y^2+1}-1)\right)\\=e^{\frac{-p^0-q^0}{4}}\exp\left(-l(\sqrt{y^2+1}-1)\right).\end{multline*} Plugging this into \eqref{zeta2abs}, we have
		 \begin{equation}\label{zeta2abs2}
	|\zeta_2|\lesssim  \frac{1}{p^0}\int_{\rth}\frac{dq}{q^0}\frac{e^{-q^0}}{g}s\sqrt{s}\Phi(g)Y(p,q),
		 \end{equation}
	where we define
	\begin{multline}\label{eq.Y}Y(p,q)\eqdef \int_{0}^\infty \frac{y^{-1-\gamma}dy}{\sqrt{y^2+1}}(y^2+\sqrt{y^2+1})^{1+\gamma/2}
	\\\times\exp\left(-l(\sqrt{y^2+1}-1)\right)\left[I_0(jy)-1 \right].\end{multline}
	For the upper-bound estimate of $Y(p,q)$ we split the region $[0,\infty)$ into two: 
	\begin{equation}\notag
	Y(p,q) = \tilde{Y}_1(p,q) + \tilde{Y}_2(p,q),
	\end{equation}   
	where $\tilde{Y}_1(p,q)$ is the integral in \eqref{eq.Y} restricted to the integration region $y\ge 1$ and then $\tilde{Y}_2(p,q)$ is the expression in \eqref{eq.Y} on the integration region $0<y<1$.

First we consider the case $\tilde{Y}_1(p,q)$  that $y\geq 1$. When $y\ge 1$, we have
$$
	\frac{y^{-1-\gamma}}{\sqrt{y^2+1}}(y^2+\sqrt{y^2+1})^{1+\gamma/2}
		\lesssim
	\frac{y^{-1-\gamma}}{\sqrt{y^2+1}}(y^2)^{1+\gamma/2}
	\lesssim\frac{y}{\sqrt{y^2+1}}.
$$ 
Therefore, on the region $y\geq 1$, using \eqref{J2.special} we have
$$
\tilde{Y}_1(p,q)\lesssim \exp(l) \int_1^\infty dy\  \frac{y}{\sqrt{y^2+1}}\exp(-l\sqrt{y^2+1})I_0(jy)
\lesssim \exp(l) J_2(l,j).
$$ 
By \eqref{J2.lemma} we then have
$$
\tilde{Y}_1(p,q)
\lesssim \exp(l) \frac{\exp(-\sqrt{l^2-j^2})}{\sqrt{l^2-j^2}}.
$$
Since $p^0-q^0\leq |p-q|$ from \eqref{p0q0.le.pq}, we have 
	\begin{equation}\label{exponential.bound.1}
	    \exp\left(\frac{p^0-q^0-|p-q|}{4}\right)\leq 1.
	\end{equation}
Thus, using \eqref{g.le.upper}, \eqref{l2j2}, \eqref{l2j2size} and \eqref{exponential.bound.1} we have 
\begin{multline*}
\tilde{Y}_1(p,q)
\lesssim \exp\left(\frac{p^0+q^0}{4}-\frac{\sqrt{s}}{4g}|p-q|\right)\frac{4g}{\sqrt{s}|p-q|}\\
	\lesssim \exp\left(\frac{q^0}{2}\right)\exp\left(\frac{p^0-q^0}{4}-\frac{|p-q|}{4}\right)\frac{4}{\sqrt{s}}\lesssim \frac{\exp\left(\frac{q^0}{2}\right)}{\sqrt{s}}.
\end{multline*} 
Now we will use $\Phi(g)\approx g^{\singS}$ from \eqref{hard}-\eqref{rho.def}.  In the hard interaction case \eqref{hard} we use \eqref{g.ge.lower} and \eqref{g.le.sqrtpq} in \eqref{zeta2abs2} to conclude that
	\begin{equation}\label{zeta2firstcase.hard}
	\left[\zeta_2\right]_{y \ge 1}\lesssim  \int_{\rth}dq\frac{\exp\left(\frac{-q^0}{2}\right)}{|p-q|}(p^0q^0)^{\frac{\singA+1}{2}}
	\lesssim(p^0)^{\frac{\singA}{2}-\frac{1}{2}}.
	\end{equation}
Then in the soft interaction case \eqref{soft}, using $b<2$, we use \eqref{s.le.pq} and \eqref{g.ge.lower} to obtain
	\begin{equation}\label{zeta2firstcase.soft}
	\left[\zeta_2\right]_{y \ge 1}\lesssim \int_{\rth}dq\frac{\exp\left(\frac{-q^0}{2}\right)}{|p-q|^{1+\singB}}(p^0q^0)^{\frac{\singB+1}{2}}
	\lesssim(p^0)^{-\frac{\singB}{2}-\frac{1}{2}}.
	\end{equation}
	Here $\left[\zeta_2\right]_{y \ge 1}$ is $\zeta_2$ restricted to the integration region ${y \ge 1}$ using the convention \eqref{convention}.
	This completes the proof for the upper bound of $\zeta_2$ when  ${y \ge 1}$.

Alternatively, using \eqref{eq.Y} we will show that $|\zeta_2|$ on $0<y<1$ is bounded uniformly from above by $(p^0)^{\frac{\singS}{2}+\frac{\gamma}{2}}.$ 
We prove this using the known Taylor expansion of the modified Bessel function of the first kind $I_0$ \cite{Gradshteyn:1702455}  as follows:
	$$I_0(jy)=\sum_{M=0}^{\infty}\frac{1}{(M!)^2}\left(\frac{jy}{2}\right)^{2M}.$$
	Now, since $y<1$, recalling \eqref{eq.Y} we have
	$$\frac{y^{-1-\gamma}}{\sqrt{y^2+1}}(y^2+\sqrt{y^2+1})^{1+\gamma/2}\lesssim y^{-1-\gamma},$$ and
	$$\exp\left(-l(\sqrt{y^2+1}-1)\right)=\exp\left(-l\frac{y^2}{\sqrt{y^2+1}+1}\right)\leq \exp\left(-l\frac{y^2}{\sqrt{2}+1}\right).$$
	Therefore, by \eqref{eq.Y}, we have
	\begin{equation*}
	\begin{split}
	\tilde{Y}_2(p,q) &\lesssim  \int_0^1 dy \ y^{-1-\gamma}
	\exp\left(-l\frac{y^2}{\sqrt{2}+1}\right)\sum_{M=1}^{\infty}\frac{1}{(M!)^2}\left(\frac{jy}{2}\right)^{2M}\\
	&\lesssim \sum_{M=1}^{\infty} \frac{1}{(M!)^2}(j/2)^{2M} \int_0^1 dy \ y^{-1-\gamma+2M}
	\exp\left(-cly^2\right),
	\end{split}
	\end{equation*} 
	where we define 
\begin{equation}
    	\label{c.def}
	c\eqdef \frac{1}{1+\sqrt{2}}.
\end{equation}
	For $M\ge 1$, we further define 
	$$
	Y_M \eqdef \int_0^1dy \ y^{-1-\gamma+2M}
	\exp\left(-cly^2\right).
	$$
	Here we take a change of variables $y\mapsto z=ly^2$ with $dz=2lydy$ and obtain
	\begin{multline}\notag
	Y_M\leq  \frac{l^{\frac{\gamma}{2}-M}}{2} \int_0^l dz \ z^{-1-\gamma/2+M}\exp\left(-cz\right) \\\leq\frac{l^{\frac{\gamma}{2}-M}}{2} \int_0^\infty dz \ z^{-1-\gamma/2+M}    	\exp\left(-cz\right)\\
	\leq\frac{l^{\frac{\gamma}{2}-M}}{2}3^{M-1} \sup_{z\in[0,\infty)} \left\{\left(\frac{z}{3}\right)^{M-1}    	\exp\left(-\frac{z}{3}\right)\right\}\int_0^\infty dz \ z^{-\gamma/2}    	\exp\left(-(c-1/3)z\right)\\
		\leq C_1 3^{M}\sup_{z\in[0,\infty)} \left\{\left(\frac{z}{3}\right)^{M-1}    	\exp\left(-\frac{z}{3}\right)\right\}l^{\frac{\gamma}{2}-M}, \end{multline}
where the constant $C_1$ is uniformly bounded since $\gamma \in(0,2)$ as
\begin{equation} \notag
    C_1 \eqdef \frac{1}{6} \int_0^\infty dz \ z^{-\gamma/2}    	\exp\left(-(c-1/3)z\right) < \infty.
\end{equation}
This holds because $c> \frac{1}{3}$ from \eqref{c.def}.  We use the Stirling formula error bounds to obtain 
	$$\sup_{z\in[0,\infty)} \left\{\left(\frac{z}{3}\right)^{M-1}    	\exp\left(-\frac{z}{3}\right)\right\}
	\leq
	\frac{1}{\sqrt{2\pi}}
	\frac{(M-1)!}{\sqrt{M-1}}
		\leq
	\frac{1}{\sqrt{4\pi}}
	\frac{M!}{\sqrt{M}},\text{ if }M\ge 2.
	$$
	Alternatively if $M=1$ we have the bound 
		$$\sup_{z\in[0,\infty)} \left\{    	\exp\left(-\frac{z}{3}\right)\right\}
\leq 1, \text{ if }M=1.$$
	Therefore we have the general bound
	\begin{equation}\label{yM.bound}
	    Y_M \leq C_1 3^M \frac{M!}{\sqrt{M}} l^{\frac{\gamma}{2}-M}, \quad M\ge 1.
	\end{equation}
We will use this bound to estimate $\left[\zeta_2\right]_{0<y<1}$ using the convention \eqref{convention}.

First we notice that using \eqref{lj} we have
	\begin{equation}\label{YM2spc}
	j^{2M} l^{\frac{\gamma}{2}-M} \leq (q^0)^{M}l^{\frac{\gamma}{2}},\end{equation} 
	where to prove \eqref{YM2spc} we used $j^2/l\leq q^0$ which follows from \eqref{g.ge.2lower} as 
	\begin{equation}\label{j2l}\frac{j^2}{l}=\frac{|p\times q|^2}{g^2(p^0+q^0)}\leq \frac{p^0q^0}{p^0+q^0}\leq q^0.
	\end{equation} 
	Now we plug \eqref{yM.bound} and \eqref{YM2spc} into \eqref{zeta2abs2} with $\tilde{Y}_2(p,q)$, to obtain
	\begin{multline*}
	\left[\zeta_2\right]_{0<y<1}
	\lesssim  \frac{1}{p^0}\int_{\rth}\frac{dq}{q^0}\frac{e^{-q^0}}{g}s\sqrt{s}\Phi(g)l^{\frac{\gamma}{2}}\sum_{M=1}^\infty \frac{1}{M!\sqrt{M}}\left(\frac{3}{4}\right)^{M}(q^0)^M
	\\
\lesssim 	 \frac{1}{p^0}\int_{\rth}\frac{dq}{q^0}\frac{e^{-q^0}}{g}s\sqrt{s}\Phi(g)l^{\frac{\gamma}{2}}\exp\left({\frac{3}{4}q^0}\right).
\end{multline*} 
We use $\Phi(g)\approx g^{\singS}$ with $-2 <\rho$ from \eqref{rho.def}.  
In the hard interaction case \eqref{hard}, we will use   \eqref{l.upper.ineq}, \eqref{g.ge.lower} and \eqref{g.le.sqrtpq} to conclude that
	\begin{equation}\label{zeta2secondcase.hard}
	\left[\zeta_2\right]_{0<y<1}\lesssim \int_{\rth}\frac{dq\ e^{-\frac{q^0}{4}}}{|p-q|}(p^0q^0)^{\frac{1}{2}+\frac{\gamma}{2} +\frac{1+\singA}{2}}\lesssim (p^0)^{\frac{\singA}{2}+\frac{\gamma}{2}}.
	\end{equation}
And in the soft interaction case \eqref{soft} we will use 	\eqref{g.ge.lower} and \eqref{s.le.pq} to obtain
	\begin{equation}\label{zeta2secondcase.soft}
	\left[\zeta_2\right]_{0<y<1}\lesssim \int_{\rth}dq\frac{e^{-\frac{q^0}{4}}}{|p-q|^{1+\singB}}(p^0q^0)^{\frac{1}{2}+\frac{\gamma}{2} +\frac{1+\singB}{2}}\lesssim (p^0)^{-\frac{\singB}{2}+\frac{\gamma}{2}},
	\end{equation}
where we recall that $1+\singB<3.$	This proves that $\left[\zeta_2\right]_{0<y<1}$ has the leading order upper bound.
\end{proof}

Thus we obtain Proposition \ref{prop.zeta0.asymptotic} by combining Lemmas \ref{lemma.zeta1} and \ref{lemma.zeta2}. In the next section, we will prove that the remainder terms  $\zetaL$ from \eqref{zetaLBy}, and $\zetaTLm$ from \eqref{def.zetaMdef}  have lower order upper bounds.  We will also prove that $\zetaTone(p)$ from \eqref{def.zetatilde1} has a lower order upper bound in Proposition \ref{prop.tildezeta.upper.qgep}.

\section{Lower order upper bound estimates}\label{sec:low order zeta1}

In this section, we study the upper bound estimates of $\zetaL$ from \eqref{zetaLBy}, $\zetaTLm$ from \eqref{def.zetaMdef} and $\zetaTone(p)$ from \eqref{def.zetatilde1}, which together form part of $\zeta_{\mathcal{K}}$ in \eqref{zetaK.def}.  Our goal will be to prove that $|\zetaL(p)|$, $|\zetaTLm(p)|$, and $|\zetaTone(p)|$ have lower order upper bounds.

\subsection{Lower order upper bound for $\zetaL(p)$} For the proof of the lower order upper bound of $|\zetaL|$ we will use the representation in \eqref{zetaLBy}. We have the following uniform asymptotic bound:

\begin{proposition}\label{prop.zetaL.asymptotic}
	Suppose $\gamma \in (0,2)$  in \eqref{angassumption}. Then for both hard \eqref{hard} and soft \eqref{soft} interactions, for \eqref{zetaLBy} when $|p|\ge 1$, we have
	$$|\zetaL(p)| \lesssim (p^0)^{\frac{\singS}{2}}.$$
This bound then automatically also holds for $|\zetaLm(p)|$ from \eqref{def.zetaMdef}.
\end{proposition}

\begin{proof}   By \eqref{zetaLBy} and the definition of $l$ and $j$ of \eqref{lj} we have
	\begin{multline}\label{zetaL.def.here}
	\zetaL \eqdef \frac{c'}{p^0}e^{\frac{p^0}{4}}\int_{\rth}\frac{dq}{q^0}\frac{e^{-\frac{3}{4}q^0}}{g}\sqrt{s}\\
	\times\int_{0}^\infty \frac{ydy}{\sqrt{y^2+1}}s_\Lambda\sigma(g_\Lambda,\theta_\Lambda)
	\exp\left(-l\sqrt{y^2+1}\right)I_0\left(jy\right) \left(\frac{s\Phi(g)g^4}{s_\Lambda  \Phi(g_\Lambda)g^4_\Lambda}-1 \right).
	\end{multline} 
	In the hard interaction case \eqref{hard}, we have $\frac{\Phi(g)}{\Phi(g_\Lambda)}=\left(\frac{g}{g_\Lambda}\right)^{a}$ with $a<2.$ Since $g\leq g_\Lambda$ from \eqref{g2y.variable},  we have $\frac{\Phi(g)}{\Phi(g_\Lambda)}\geq \frac{g^2}{g^2_\Lambda}$. Then this  implies
	$$
	\left|\frac{s\Phi(g)g^4}{s_\Lambda  \Phi(g_\Lambda)g^4_\Lambda} - 1\right|=1-\frac{s\Phi(g)g^4}{s_\Lambda  \Phi(g_\Lambda)g^4_\Lambda} \leq 1-\frac{sg^6}{s_\Lambda  g^6_\Lambda}=\frac{s_\Lambda g^6_\Lambda-sg^6}{s_\Lambda g^6_\Lambda}.
	$$
	Further note that we have 
	\begin{multline*}
	s_\Lambda g^6_\Lambda-sg^6=(g^8_\Lambda-g^8)+4(g^6_\Lambda-g^6)\\= (g^2_\Lambda-g^2)\left((g_\Lambda^4+g^4)(g^2_\Lambda+g^2)+4g^4_\Lambda+4g^2_\Lambda g^2+4g^4\right)\\\lesssim \frac{s}{2}(\sqrt{y^2+1}-1)
	g^4_\Lambda s_\Lambda,
	\end{multline*}
	since $g^2_\Lambda-g^2=\frac{s}{2}(\sqrt{y^2+1}-1)$ from \eqref{g2y.variable}, $s_\Lambda\eqdef g^2_\Lambda+4$ and again $g\leq g_\Lambda$.
	Therefore, we have 
\begin{equation}\label{difference.estimate.here}
    	\left|\frac{s\Phi(g)g^4}{s_\Lambda  \Phi(g_\Lambda)g^4_\Lambda} - 1\right|\leq\frac{s_\Lambda g^6_\Lambda-sg^6}{s_\Lambda g^6_\Lambda}\lesssim \frac{\frac{s}{2}(\sqrt{y^2+1}-1)}{g^2_\Lambda}.
\end{equation}
	This is the main estimate  for this difference in the hard interaction case.
	
We now consider the same estimate in the soft interaction case \eqref{soft}.   Since $g\leq g_\Lambda$ and $\frac{\Phi(g)}{\Phi(g_\Lambda)}=\left(\frac{g}{g_\Lambda}\right)^{-b}$ with $b\in[\gamma,2),$ then we have $\frac{\Phi(g)}{\Phi(g_\Lambda)}\geq 1$. Then $b\in[\gamma,2)$ 
	further implies
	$$1-\frac{s\Phi(g)g^4}{s_\Lambda  \Phi(g_\Lambda)g^4_\Lambda} \leq 1-\frac{sg^4}{s_\Lambda  g^4_\Lambda}=\frac{s_\Lambda g^4_\Lambda-sg^4}{s_\Lambda g^4_\Lambda}.
	$$
In this case we also have
\begin{multline*}
	s_\Lambda g^4_\Lambda-sg^4=(g^6_\Lambda-g^6)+4(g^4_\Lambda-g^4)= (g^2_\Lambda-g^2)\left(g^4_\Lambda+g^2_\Lambda g^2+g^4+4g^2_\Lambda+4g^2\right)\\=(g^2_\Lambda-g^2)\left(g^4_\Lambda+g^2_\Lambda g^2+g^4+4g^2_\Lambda+4g^2\right)
	\leq \frac{s}{2}(\sqrt{y^2+1}-1)(3g^4_\Lambda+8g^2_\Lambda)\\
	\lesssim
	\frac{s}{2}(\sqrt{y^2+1}-1)g^2_\Lambda s_\Lambda,
\end{multline*}
because again $g^2_\Lambda-g^2=\frac{s}{2}(\sqrt{y^2+1}-1)$ from \eqref{g2y.variable}.  Therefore, we have
	\begin{equation}\label{upperbound.allcases.zeta1}
	\left|\frac{s\Phi(g)g^4}{s_\Lambda  \Phi(g_\Lambda)g^4_\Lambda} - 1\right|\leq\frac{s_\Lambda g^4_\Lambda-sg^4}{s_\Lambda g^4_\Lambda}\lesssim \frac{\frac{s}{2}(\sqrt{y^2+1}-1)}{g^2_\Lambda}.
	\end{equation}
Note that in both the hard interaction case and the soft interaction case the final upper bounds are the same in  \eqref{difference.estimate.here} and \eqref{upperbound.allcases.zeta1}.

	In both cases, then plugging \eqref{difference.estimate.here} and \eqref{upperbound.allcases.zeta1} into \eqref{zetaL.def.here}  we have 
	\begin{equation}\label{zetaLm.upper.first}
	|\zetaL|  \lesssim  \frac{1}{p^0}e^{\frac{p^0}{4}}\int_{\rth}\frac{dq}{q^0}\frac{e^{-\frac{3}{4}q^0}}{g}s^{1/2}
K_2(p,q),
	\end{equation}
where we define $K_2 = K_2(p,q)$ by
	\begin{equation}\label{K2definition.appendix}
	K_2\eqdef  \int_{0}^\infty \frac{ydy}{\sqrt{y^2+1}}s_\Lambda\sigma(g_\Lambda,\theta_\Lambda)
	\exp\left(-l\sqrt{y^2+1}\right)I_0\left(jy\right) \frac{\frac{s}{2}(\sqrt{y^2+1}-1)}{g^2_\Lambda}.
	\end{equation}
	We will split into two cases:
	 $y\leq 1$ and $y>1$. We write $K_2=K_{2,\le 1}+K_{2,\ge 1}$ below where  $K_{2,\le 1}$ and $ K_{2,\ge 1}$ denote $K_2$ on $y\le 1$ and $y\ge 1$, respectively.

First let us generally estimate the kernel.  We will use the product form \eqref{kernel.product} with the estimates \eqref{hard}-\eqref{soft}-\eqref{rho.def} to obtain
\begin{equation}\notag
    s_\Lambda\sigma(g_\Lambda,\theta_\Lambda) \frac{\frac{s}{2}(\sqrt{y^2+1}-1)}{g^2_\Lambda}
    \approx s_\Lambda g_\Lambda^{\rho-2} s \sigma_0(\theta_\Lambda) (\sqrt{y^2+1}-1) 
    \approx   \frac{s_\Lambda g_\Lambda^{\rho} s \sigma_0(\theta_\Lambda) y^2 }{g_\Lambda^{2}(\sqrt{y^2+1}+1)}.
\end{equation}
Additionally using \eqref{ang.equiv} with \eqref{glambda.calc} we have 
\begin{equation}\notag
\sigma_0(\theta_\Lambda)
\approx
\left(\frac{s}{ g_\Lambda^2} \frac{y^2}{2(\sqrt{y^2+1}+1)}\right)^{-1-\gamma/2}
\lesssim  y^{-2-\gamma}(\sqrt{1+y^2})^{1+\gamma/2} \left(\frac{ g_\Lambda^2}{s}\right)^{1+\gamma/2}
\end{equation}
We plug this into the previous estimate to obtain
\begin{multline}\notag
s_\Lambda\sigma(g_\Lambda,\theta_\Lambda) \frac{\frac{s}{2}(\sqrt{y^2+1}-1)}{g^2_\Lambda}
\lesssim  
y^{-\gamma}(1+y^2)^{\gamma/4}
\frac{s_\Lambda g_\Lambda^{\rho} s }{g_\Lambda^{2}}\left(\frac{ g_\Lambda^2}{s}\right)^{1+\gamma/2}
\\
\lesssim  
y^{-\gamma}(1+y^2)^{\gamma/4}
s_\Lambda g_\Lambda^{\rho} \left(\frac{ g_\Lambda^2}{s}\right)^{\gamma/2}.
\end{multline}
We conclude from \eqref{ineq.gL.here} and the above that in general we have
\begin{equation}\notag
s_\Lambda\sigma(g_\Lambda,\theta_\Lambda) \frac{\frac{s}{2}(\sqrt{y^2+1}-1)}{g^2_\Lambda}
\lesssim  
y^{-\gamma}(1+y^2)^{\gamma/2}
s_\Lambda g_\Lambda^{\rho} 
\lesssim  
y^{-\gamma}(1+y^2)^{(1+\gamma)/2}
s g_\Lambda^{\rho}.
\end{equation}
In particular, recalling \eqref{hard}-\eqref{soft}-\eqref{rho.def} and using \eqref{difference.estimate.here}, \eqref{upperbound.allcases.zeta1}, and \eqref{ineq.gL.here}, then in general we have
\begin{equation}\label{general.kernel.est.here}
s_\Lambda\sigma(g_\Lambda,\theta_\Lambda)
	\left|\frac{s\Phi(g)g^4}{s_\Lambda  \Phi(g_\Lambda)g^4_\Lambda} - 1\right|
\lesssim  \mathbb{S} ~ 
y^{-\gamma}(1+y^2)^{1+\gamma/2}, \quad \forall 0 \le y \le \infty.
\end{equation}
This holds in particular since $a<2$ in \eqref{hard}.  Here we define
	\begin{equation}\label{S.hard.soft.def}
	\begin{split}\mathbb{S}&\eqdef s^{1+\frac{a}{2}} \text{  for hard interactions \eqref{hard},}\\
	&\eqdef sg^{-b}\text{  for soft interactions \eqref{soft}.}\end{split}\end{equation}
These are the specific estimates that we will use on the kernel of  \eqref{K2definition.appendix}.

	Now we return to estimating \eqref{K2definition.appendix} on the region when $y\leq 1$.  
	Then using the above calculations we have the following bound for $\left|K_{2,\le 1}\right|$:
	\begin{multline}\notag
	|K_{2,\le 1}|
	\lesssim
	\mathbb{S}\int_0^{1}dy\  y^{1-\gamma}
	(1+y^2)^{(1+\gamma)/2}
	\exp\left(-l\sqrt{y^2+1}\right) I_0(j y)
	\\
	\lesssim
\mathbb{S}\int_0^{1}dy\  y^{1-\gamma}
\exp\left(-l\sqrt{y^2+1}\right) I_0(j y)
	\lesssim
\mathbb{S}\bar{K}_\gamma(l,j),
	\end{multline}  
where we used $(1+y^2)^{(1+\gamma)/2}\lesssim 1$ as $y\in (0,1)$, and $\gamma>0$.  	Since $\gamma\in (0,2)$, this integral converges.  In the last upper bound we used \eqref{int.Kgamma}. From \eqref{smally.lemma} we conclude
	\begin{equation}\label{smally}
|K_{2,\le 1}|
	\lesssim
	\mathbb{S} \exp\left(-\sqrt{l^2-j^2}\right).
	\end{equation}
	This completes our estimate on the region when $y \le 1$.

	On the other region when $y>1$, using \eqref{general.kernel.est.here} and \eqref{S.hard.soft.def},   for the integral defined in \eqref{K2definition.appendix}  we have for both hard and soft interactions, that
$$
|K_{2,\ge 1}|
\lesssim
\mathbb{S}\int_{1}^\infty\ dy\ y(y^2+1)^{1/2}\exp\left(-l\sqrt{y^2+1}\right) I_0(j y)
\lesssim \mathbb{S} \tilde{K}_2(l,j).
$$
Here $\tilde{K}_2(l,j)$ is defined in \eqref{tildeK2def}.  The formula for $\tilde{K}_2$ is \eqref{k2lj.lemma} 
and we have
\begin{equation}\label{k2lj}
\tilde{K}_2(l,j)
	\lesssim (\sqrt{l^2-j^2})^{-5}\exp(-\sqrt{l^2-j^2}) (l^2-j^2+1)l^2.
\end{equation}
By \eqref{l2j2}, we have $l^2-j^2=\frac{s}{16g^2}|p-q|^2.$
Thus, we obtain
\begin{multline}\notag\tilde{K}_2(l,j)\lesssim \frac{l^2}{(\sqrt{l^2-j^2})^3} \left(1+\frac{1}{l^2-j^2}\right)\exp(-\sqrt{l^2-j^2}) \\
 	\lesssim
 	\frac{(p^0+q^0)^2}{16}\left(\frac{4g}{\sqrt{s}|p-q|}\right)^3\left(1+\left(\frac{4g}{\sqrt{s}|p-q|}\right)^2\right)\exp(-\sqrt{l^2-j^2}).\end{multline}
 	We point out that due to \eqref{g.le.upper} the above is not singular when $|p-q|=0$.

 	Note that using \eqref{s.ge.g2} and \eqref{g.le.upper} we have
 \begin{equation}\label{tildeK2size0} 1+\left(\frac{4g}{\sqrt{s}|p-q|}\right)^2\le 1+\frac{16}{s}\lesssim 1.  
\end{equation} 
Also using $g\le \sqrt{s}$, which follows from \eqref{s.ge.g2}, we have
 \begin{multline}\label{tildeK2size}\tilde{K}_2(l,j)\lesssim  \frac{(p^0+q^0)^2}{16}\left(\frac{4g}{\sqrt{s}|p-q|}\right)^3\exp(-\sqrt{l^2-j^2})\\
 	\lesssim  	\frac{(p^0+q^0)^2}{16}\left(\frac{4}{|p-q|}\right)^3\exp(-\sqrt{l^2-j^2})\\
\lesssim  	\frac{(p^0+q^0)^2}{|p-q|^3}\exp(-\sqrt{l^2-j^2}).
 \end{multline}
 We will use estimate \eqref{tildeK2size} to control the size of $|K_{2,\ge 1}|$ below.

We will now to use the region $|q|\le \frac{1}{2}|p|^{1/m}$ and $|p| \ge 1$  to complete our estimate of $|K_{2,\ge 1}|$.  Then later we will do separate estimates on the complementary region: $|q|\ge \frac{1}{2}|p|^{1/m}$.  Now since $|q|\le \frac{1}{2}|p|^{1/m}$, $m>1$ and $|p| \ge 1$, 
then we have 
\begin{equation}\label{pminusq bound}\frac{p^0}{4}\le \frac{|p|}{2}\le |p-q|\le \frac{3}{2}|p|,
\end{equation}
and
\begin{equation}\label{q0 bound}1\le q^0 \le2 (p^0)^{1/m}.\end{equation}
Thus we have
$$\frac{(p^0+q^0)^2}{|p-q|^3}\lesssim \frac{1}{p^0}.$$
Hence we obtain from \eqref{tildeK2size} that 
 \begin{equation}
     \label{tildeK2size.ref2}
 \tilde{K}_2(l,j)
\lesssim  	\frac{1}{p^0}\exp(-\sqrt{l^2-j^2}).
 \end{equation}
We therefore conclude from \eqref{tildeK2size.ref2} that 
\begin{equation}
	\label{largey}
		|K_{2,\ge 1}|\lesssim \mathbb{S}\frac{1}{p^0}\exp\left(-\sqrt{l^2-j^2}\right).
	\end{equation}
	By \eqref{S.hard.soft.def}, (\ref{smally}) and (\ref{largey}), using \eqref{convention}  we finally obtain
	\begin{multline*}
	[|\zetaL(p)|]_{|q|\le \frac{1}{2}|p|^{1/m}} \lesssim \frac{1}{p^0}e^{\frac{p^0}{4}}\int_{\qlep} \frac{dq}{q^0}\frac{\sqrt{s}}{g}e^{-\frac{3}{4}q^0}\mathbb{S}\exp\left(-\sqrt{l^2-j^2}\right)\\
\lesssim \frac{1}{p^0}\int_{\qlep} \frac{dq}{q^0}e^{-\frac{q^0}{2}}\frac{\sqrt{s}}{g}\mathbb{S}\exp\left(\frac{p^0-q^0}{4}-\sqrt{l^2-j^2}\right).
	\end{multline*} 
	We will use the inequality above to obtain the final upper bounds.

	In the hard interaction case \eqref{hard}, we use \eqref{s.le.pq}, \eqref{g.ge.lower} and \eqref{l2j2size} to obtain
	\begin{equation}\label{zeta2final.hard}
	[|\zetaL(p)|]_{|q|\le \frac{1}{2}|p|^{1/m}}\lesssim \int_{\rth}dq\ \frac{(p^0q^0)^{1+\frac{\singA}{2}}}{|p-q|}e^{-\frac{q^0}{2}}\exp\left(\frac{p^0-q^0-|p-q|}{4}\right).
	\end{equation}
	Therefore, from \eqref{exponential.bound.1} we have
$$	[|\zetaL(p)|]_{|q|\le \frac{1}{2}|p|^{1/m}}\lesssim(p^0)^{\frac{\singA}{2}}.$$
	In the soft interaction case \eqref{soft}, we use \eqref{g.ge.lower} and \eqref{l2j2size} to obtain
	\begin{multline}\label{zeta2final.soft}
	[|\zetaL(p)|]_{|q|\le \frac{1}{2}|p|^{1/m}} \\ 
	\lesssim \int_{	{|q|\le \frac{1}{2}|p|^{1/m}}}dq\ \frac{(p^0q^0)^{\frac{\singB}{2}+1}}{|p-q|^{1+\singB}} e^{-\frac{q^0}{2}}\exp\left(\frac{p^0-q^0-|p-q|}{4}\right).
	\end{multline}
	By \eqref{exponential.bound.1}, since $\singB<2$, we then have
	\begin{equation*}	[|\zetaL(p)|]_{|q|\le \frac{1}{2}|p|^{1/m}} \lesssim(p^0)^{\frac{\singB}{2}+1}\int_{\rth}dq\ |p-q|^{-1-\singB} e^{-\frac{q^0}{2}}(q^0)^{\frac{\singB}{2}+1}
	\lesssim (p^0)^{-\frac{\singB}{2}}.
	\end{equation*} 
This completes the desired estimates on the region $\qlep$

Next we perform the estimates on the complementary region where $|q|\ge \frac{1}{2}|p|^{1/m}$, $|p| \ge 1$ and $m>1$.   Therefore, in this region we have
\begin{equation}\notag
q^0 \ge |q|  \ge \frac{1}{2} \left(|p|^2 \right)^{\frac{1}{2m}}
\ge 
\frac{1}{2}\left(\frac{1}{2}\right)^{\frac{1}{2m}} \left(p^0 \right)^{\frac{1}{m}}.
\end{equation}
Then, for some $c=c_m>0$, we have additional exponential decay from
\begin{equation}\label{exponential.more}
e^{-q^0} = e^{-q^0/2} e^{-q^0/2}\le e^{-q^0/2} e^{-c(p^0)^{1/m}},
\end{equation}
Then with \eqref{exponential.more} we have exponential decay in $p^0$.

Now we need to replace the estimates on $\tilde{K}_2(l,j)$ above, which is defined in \eqref{tildeK2def}.  Recalling the estimates \eqref{k2lj} and \eqref{tildeK2size0}, instead of \eqref{tildeK2size} we use \eqref{s.ge.g2} and \eqref{g.le.upper} to obtain
 \begin{multline}\label{tilde.k2.estimate.region}
 \tilde{K}_2(l,j)\lesssim  \frac{(p^0+q^0)^2}{16}\left(\frac{4g}{\sqrt{s}|p-q|}\right)^3\exp(-\sqrt{l^2-j^2})\\
 	\lesssim (p^0+q^0)^2
 	\exp(-\sqrt{l^2-j^2}).
 \end{multline}
 We conclude from the above estimate, recalling also \eqref{S.hard.soft.def}, that
 \begin{equation}\notag
		|K_{2,\ge 1}|\lesssim \mathbb{S}(p^0+q^0)^2
 	\exp(-\sqrt{l^2-j^2}).
	\end{equation}
Then by \eqref{zetaLm.upper.first}, \eqref{K2definition.appendix}, \eqref{S.hard.soft.def}, (\ref{smally}) and the above,  we further obtain
	\begin{multline}\notag
	[|\zetaL(p)|]_{\qgep} \lesssim \frac{1}{p^0}e^{\frac{p^0}{4}}\int_{\qgep}\frac{dq}{q^0}\frac{\sqrt{s}}{g}e^{-\frac{3}{4}q^0}\mathbb{S}(p^0+q^0)^2\exp\left(-\sqrt{l^2-j^2}\right)\\
\lesssim \frac{1}{p^0}\int_{\qgep}\frac{dq}{q^0}e^{-\frac{q^0}{2}}\frac{\sqrt{s}}{g}\mathbb{S}(p^0+q^0)^2\exp\left(\frac{p^0-q^0}{4}-\sqrt{l^2-j^2}\right)
\\
\lesssim p^0e^{-c(p^0)^{1/m}}\int_{\qgep } dq ~ q^0 e^{-\frac{q^0}{4}}\frac{\sqrt{s}}{g}\mathbb{S},
	\end{multline} 	
where in the last inequality we used \eqref{p0.plus.q0.le.p0q0}, \eqref{l2j2size}, 	\eqref{exponential.bound.1} and \eqref{exponential.more}. Then from the same procedures we used to prove \eqref{zeta2final.hard} and \eqref{zeta2final.soft}, using the exponential decay in $p^0$ above, we obtain for some uniform $c'>0$ that
	\begin{equation}\label{zetaLqgep}
	[|\zetaL(p)|]_{\qgep} 
\lesssim e^{-c'(p^0)^{1/m}}.
	\end{equation} 	
Combining the previous estimates, this completes the proof.
\end{proof}

This completes the proof of the lower order upper bound estimates for $\zetaL$ from \eqref{zetaLBy}. Next, we prove that $\zetaTLm$ from \eqref{def.TzetaMdef} has a lower order upper bound.

\subsection{Lower order upper bound for $\zetaTLm$}
We now prove the following proposition:

\begin{proposition}\label{prop.zetaL.asymptoticnew}
	Suppose $\gamma \in (0,2)$  in \eqref{angassumption}, and recall \eqref{rho.def}.  For any given small $\varepsilon>0$, assume that $m$ is sufficiently large such that  
	$\frac{|\singS|+8}{2m}\le \varepsilon.$  Then for both hard \eqref{hard} and soft \eqref{soft} interactions there exists a finite constant $C_\varepsilon>0$ such that for \eqref{def.TzetaMdef} we have the following uniform asymptotic estimate
	$$\left|\zetaTLm(p)\right| \leq C_\varepsilon (p^0)^{\frac{\singS}{2}+\varepsilon}.$$
\end{proposition}

We recall \eqref{eq:tildezetanew1L} and \eqref{def.TzetaMdef} with \eqref{convention}.  Then we use the following representation in this section (implicitly assuming $|p| \ge 1$): 
\begin{multline}\notag
	\zetaTLm(p)
	=\frac{c'}{ p^0}\int_{\qlep}\frac{dq}{q^0}\frac{e^{-q^0}\sqrt{s}}{g}\int_{0}^\infty \frac{ydy}{\sqrt{y^2+1}}s_\Lambda\sigma(g_\Lambda,\theta_\Lambda)\left(1-\frac{s\Phi(g)g^4}{s_\Lambda  \Phi(g_\Lambda)g^4_\Lambda} \right)\\\times \left( \exp(2  l -2 l\sqrt{y^2+1})I_0(2jy)
-\exp(  l - l\sqrt{y^2+1})I_0(jy)\right)
	\eqdef \zetaLTone-\zetaLTtwo.\end{multline}
The splitting of $\zetaTL$ into $\zetaLTone$ and $\zetaLTtwo$ allows us to realize $\zetaLTtwo=-\zetaLm$ from \eqref{def.zetaMdef} with \eqref{zetaLBy} and the lower order upper bound estimate for $\zetaL$ was already given in Proposition \ref{prop.zetaL.asymptotic}.

\begin{proof}
Based on the above discussion, in this proof we only need to give the asymptotic upper bound for $\zetaLTone$.  We start with
\begin{multline}\notag
	\zetaLTone =\frac{c'}{ p^0}\int_{\qlep}\frac{dq}{q^0}\frac{e^{-q^0}\sqrt{s}}{g}\int_{0}^\infty \frac{ydy}{\sqrt{y^2+1}}s_\Lambda\sigma(g_\Lambda,\theta_\Lambda)\left(1-\frac{s\Phi(g)g^4}{s_\Lambda  \Phi(g_\Lambda)g^4_\Lambda} \right)\\\times  \exp(2  l -2 l\sqrt{y^2+1})I_0(2jy).
\end{multline} 
We recall that we have the kernel estimate 
 \eqref{general.kernel.est.here} with the notation \eqref{S.hard.soft.def}.

Thus when $y\le 1$, since $0<\gamma<2$, we have
\begin{equation}\notag
	[ \zetaLTone]_{ y\le1} \lesssim\frac{1}{ p^0}\int_{\qlep}\frac{dq}{q^0}\frac{e^{-q^0}s^{1/2}}{g}\exp(2l)\bar{K}_\gamma(2l,2j) \mathbb{S},
\end{equation}
where we defined $\bar{K}_\gamma(l,j)$ in \eqref{int.Kgamma}.
In particular from \eqref{smally.lemma} we have
\begin{equation}\notag
\bar{K}_\gamma(2l,2j) \lesssim \exp (-\sqrt{(2l)^2-(2j)^2}).\end{equation} 
Thus, when $y\le 1$, we have
\begin{equation}\label{smally.zetaL}
	[ \zetaLTone]_{ y\le1} \lesssim\frac{1}{ p^0}\int_{\qlep}\frac{dq}{q^0}\frac{e^{-q^0}s^{1/2}}{g}\exp (2l-\sqrt{4l^2-4j^2})\mathbb{S}.
\end{equation}
Above we are using the convention from \eqref{convention}.

 On the other hand, if $y\ge 1$, we again use the kernel estimate 
 \eqref{general.kernel.est.here} with \eqref{S.hard.soft.def}.  Then, for both hard and soft interactions, we have
 \begin{multline}\notag
 	[ \zetaLTone]_{ y\ge1} \lesssim\frac{1}{ p^0}\int_{\qlep} \frac{dq}{q^0}\frac{e^{-q^0}s^{1/2}}{g}\exp(2l)\mathbb{S}\\\times\int_1^\infty dy\  y(y^2+1)^{1/2}  \exp(-2 l\sqrt{y^2+1})I_0(2jy).
 \end{multline}
 Then note that 
 $$\int_1^\infty dy\  y(y^2+1)^{1/2}  \exp( -2 l\sqrt{y^2+1})I_0(2jy) \le \tilde{K}_2(2l,2j),$$ where $\tilde{K_2}$ is defined in \eqref{tildeK2def}. 
 Then by \eqref{tildeK2size.ref2}, on the region $|q|\le \frac{1}{2}|p|^{1/m}$, we have
\begin{equation}\label{kernel.region.estimate}
    \tilde{K}_2(2l,2j)\lesssim \frac{1}{p^0}\exp(-\sqrt{4l^2-4j^2}).
\end{equation}
Hence if $y\ge 1$, we have
 \begin{equation}\label{largey.zetaL}
 	[ \zetaLTone ]_{ y\ge1} \lesssim\frac{1}{ p^0}\int_\qlep \frac{dq}{q^0}\frac{e^{-q^0}s^{1/2}}{gp^0}\mathbb{S}\exp(2l-\sqrt{4l^2-4j^2}).
 \end{equation}
 Thus, combining \eqref{smally.zetaL} and \eqref{largey.zetaL}, we obtain \begin{equation}\notag
 	 \zetaLTone \lesssim \frac{1}{p^0}\int_\qlep \frac{dq}{q^0}\ \frac{e^{-q^0}s^{1/2}}{g}\mathbb{S}\exp(2l-\sqrt{4l^2-4j^2}).
 \end{equation}
 We will now split this estimate into the hard \eqref{hard} and soft \eqref{soft} interaction cases.

 In the hard interaction case \eqref{hard}, we use \eqref{g.ge.lower}, \eqref{s.le.pq} and \eqref{l2j2size} to obtain
 \begin{equation}\notag
 	\zetaLTone 
 	\lesssim \int_\qlep dq\ |p-q|^{-1}(p^0q^0)^{1+\frac{\singA}{2}}\exp\left(\frac{p^0-q^0-|p-q|}{2}\right).
 \end{equation}
 Then by \eqref{pminusq bound}, \eqref{q0 bound}, and \eqref{exponential.bound.1}, we have
 \begin{multline}\label{zetaLfinal.hard}
 	 \zetaLTone 
 	\lesssim \int_\qlep dq\ |p|^{-1}(p^0)^{\left(1+\frac{\singA}{2}\right)\left(1+\frac{1}{m}\right)}\\\lesssim |p|^{-1+\frac{3}{m}}(p^0)^{\left(1+\frac{\singA}{2}\right)\left(1+\frac{1}{m}\right)}\lesssim (p^0)^{\frac{\singA}{2}+\frac{\singA+8}{2m}},
 \end{multline}
since $|p|\ge 1.$ Then for any given small $\varepsilon>0$, we choose $m$ sufficiently large such that  
 $\frac{\singA+8}{2m}\le \varepsilon.$ This yields Proposition \ref{prop.zetaL.asymptoticnew} in the  hard interaction case on the region $|q|\le \frac{1}{2}|p|^{1/m}$.

 In the soft interaction case \eqref{soft}, we use \eqref{g.ge.lower} and \eqref{l2j2size} to obtain
 \begin{equation}\notag
 	 \zetaLTone 
 	\lesssim \int_\qlep dq\ |p-q|^{-1-\singB} (p^0q^0)^{\frac{\singB}{2}+1}\exp\left(\frac{p^0-q^0-|p-q|}{2}\right).
 \end{equation} 
 Now we use \eqref{pminusq bound}, \eqref{q0 bound}, and \eqref{exponential.bound.1} to obtain for $|p|\ge 1$ that
 \begin{multline}\label{zetaLfinal.soft}
 	 \zetaLTone
	\lesssim \int_\qlep dq\ |p|^{-1-\singB}(p^0)^{\left(1+\frac{\singB}{2}\right)\left(1+\frac{1}{m}\right)}\\\lesssim |p|^{-1-\singB+\frac{3}{m}}(p^0)^{\left(1+\frac{\singB}{2}\right)\left(1+\frac{1}{m}\right)}\lesssim (p^0)^{-\frac{\singB}{2}+\frac{\singB+8}{2m}}.
\end{multline}
For any given small $\varepsilon>0$, we choose $m$ sufficiently large such that  
$\frac{\singB+8}{2m}\le \varepsilon.$ 
This yields Proposition \ref{prop.zetaL.asymptoticnew} in the soft interaction case.  This completes the proof.
 \end{proof}

\subsection{Low-order upper-bound for $\zetaTone$} Lastly, we introduce the following proposition on the lower-order upper bound estimate for $|\zetaTone|$ from \eqref{def.zetatilde1}.
\begin{proposition}\label{prop.tildezeta.upper.qgep}
	Suppose $\gamma \in (0,2)$ and $m>0.$ Then for both hard \eqref{hard} and soft \eqref{soft} interactions, for some $c>0$, we have the uniform upper bound for \eqref{def.zetatilde1}:
	$$\left|\zetaTone(p)\right| \lesssim e^{-c(p^0)^{1/m}},$$ 
\end{proposition}

\begin{proof}
Note that on $\qgep$, then we will prove that each decomposed piece $\zeta_0$ from \eqref{zeta0B} and $\zeta_L$ from \eqref{zetaLBy} of $\tilde{\zeta}$ is lower order as above.  

In Section \ref{sec:fullsharpupper zeta0}, in both \eqref{zeta.use.hard.too} and \eqref{k.ge.four.est.zeta}, if we restrict the domain to the case $|q|\ge \frac{1}{2}|p|^{1/m}$, then we have the bound \eqref{exponential.more}
for some uniform $c>0$.  Therefore, $\zeta_1$ in this subregion is lower order in $p^0$, as it has additional exponential decay $e^{-c(p^0)^{1/m}}$. Similarly, for $\zeta_2$, in \eqref{zeta2firstcase.hard}, \eqref{zeta2firstcase.soft}, \eqref{zeta2secondcase.hard}, and \eqref{zeta2secondcase.soft}, can again use \eqref{exponential.more}
on the region $|q|\ge \frac{1}{2}|p|^{1/m}$. Thus, $\zeta_2$ is also lower order in this region, and hence $\zeta_0$ is lower order when $\qgep$.

On the other hand for $\zeta_L$ from \eqref{zetaLBy}, we have \eqref{zetaLqgep} in Section \ref{sec:low order zeta1} which is exactly the desired estimate.  Thus, $\zeta_L$ is also lower order in $p^0$, and hence $\tilde{\zeta}$ is also lower order in this sub-region.  This completes the proof.\end{proof}

\appendix

 \section{Collision frequency multiplier derivation}\label{sec:derivation}

We now explain a derivation of an alternative form of $\tilde{\zeta}(p)$ from \eqref{tildezeta}, and give the new decomposition of  $\tilde{\zeta}(p)$ that has been explained in Section \ref{sec:main.decomp}.

\subsection{Derivation of a new representation of $\tilde{\zeta}(p)$}
For a fixed $p\in\rth$, recalling \eqref{tildezeta}, we would like to have an alternative representation of the following integral:
\begin{equation*}
I\eqdef -\tilde{\zeta}(p)=	\int_{\rth\times\mathbb{S}^2} v_{\text{\o}} \sigma(g,\theta) \sqrt{J(q)}\left(\sqrt{J(q')}-\sqrt{J(q)}\right)dqd\omega
\eqdef I_{gain}-I_{loss}.
\end{equation*}
Initially, suppose that $\int_{\mathbb{S}^2} d\omega\   |\sigma_0(\cos\theta)| <\infty$ and that
$$
\int_{\mathbb{S}^2} d\omega\   \sigma_0(\cos\theta)=0.
$$
Then, under that condition, the loss term vanishes $I_{loss}=0$ and we obtain
\begin{equation}\label{Igainrepresentation}
I=I_{gain}=\int_{\rth}dq\int_{\mathbb{S}^2}d\omega\  v_{\text{\o}} \sigma(g,\theta)\sqrt{J(q)}\sqrt{J(q')}.
\end{equation}
By recovering the delta function involving the energy-momentum convervation laws, we obtain another representation of $I$:
$$I= \frac{1}{p^0}\int_{\rth}\frac{dq}{q^0}\int_{\rth}\frac{dp'}{p'^0}\int_{\rth}\frac{dq'}{q'^0}s\sigma(g,\theta)\delta^{(4)}(p'^\mu+q'^\mu-p^\mu-q^\mu)\sqrt{J(q)}\sqrt{J(q')}.$$
Here $g=g(p^\mu,q^\mu)$, $s=g^2+4$, $\bar{g}\eqdef g(p^\mu,p'^\mu)=g(q^\mu,q'^\mu)$, $\tilde{g} = g(p'^\mu,q^\mu)$, and
$$
\cos\theta=2\frac{\tilde{g}^2}{g^2}-1,
$$
by \eqref{cos}.
We further {\it claim} that
\begin{equation}\label{gg}
g^2=\tilde{g}^2-\frac{1}{2}(p^\mu+q'^\mu)(p'_\mu+q_\mu-p_\mu-q'_\mu).
\end{equation} Let $\tilde{s}\eqdef \tilde{g}^2+4$. Then (\ref{gg}) is equivalent to
\begin{multline*}
g^2=\tilde{g}^2-\frac{1}{2}\tilde{s}-\frac{1}{2}(p^\mu+q'^\mu)(p'_\mu+q_\mu)\\
=\frac{1}{2}\tilde{g}^2-2-\frac{1}{2}(p^\mu+q'^\mu)(p'_\mu+q_\mu)\\
=\frac{1}{2}\tilde{g}^2+g^2+2p^\mu q_\mu-\frac{1}{2}(p^\mu+q'^\mu)(p'_\mu+q_\mu).
\end{multline*}
Thus we prove (\ref{gg}) by showing that
$$\frac{1}{2}\tilde{g}^2+2p^\mu q_\mu-\frac{1}{2}(p^\mu+q'^\mu)(p'_\mu+q_\mu)=0.$$
Expanding the left-hand side of this equation, we obtain$$
-p^\mu q'_\mu-1+2p^\mu q_\mu-\frac{1}{2}p^\mu p'_\mu-\frac{1}{2}q'^\mu p'_\mu-\frac{1}{2}p^\mu q_\mu-\frac{1}{2}q'^\mu q_\mu.$$
By the result of the conservation laws $p^\mu+q^\mu=p'^\mu+q'^\mu$, we have $p^\mu q_\mu=p'^\mu q'_\mu$ and $p'^\mu q_\mu=p^\mu q'_\mu$. Therefore, we obtain $$
-1+p^\mu q_\mu-\frac{1}{2}p^\mu p'_\mu-\frac{1}{2}p^\mu q'_\mu-\frac{1}{2}p'^\mu q_\mu-\frac{1}{2}q'^\mu q_\mu,$$ which is equal to
$$-1+p^\mu q_\mu-\frac{1}{2}(p^\mu+q^\mu)(p'_\mu+q'_\mu)=-1+p^\mu q_\mu+\frac{1}{2}s=0.$$ This finishes the proof of the  {\it claim} (\ref{gg}).

By exchanging $q$ and $q'$, we have
$$I= \frac{1}{p^0}\int_{\rth}\frac{dq}{q^0}\sqrt{J(q)}\int_{\rth}\frac{dp'}{p'^0}\int_{\rth}\frac{dq'}{q'^0}\sqrt{J(q')}\tilde{s}\sigma(\tilde{g},\theta')\delta^{(4)}(p'^\mu+q^\mu-p^\mu-q'^\mu),$$
where the angle $\theta'$ is now defined as
$$\cos\theta'\eqdef 2\frac{g^2}{\tilde{g}^2}-1,$$ 
and
\begin{equation}\notag	\tilde{g}^2=g^2-\frac{1}{2}(p^\mu+q^\mu)(p'_\mu+q'_\mu-p_\mu-q_\mu).\end{equation} 
We have the new argument in the delta function and $\tilde{s}\eqdef \tilde{g}^2+4$.

We now define the functional $i(p,q)$ as
\begin{equation}\label{i}
i(p,q)\eqdef \frac{1}{p^0q^0} \int_{\rth}\frac{dp'}{p'^0}\int_{\rth}\frac{dq'}{q'^0}\sqrt{J(q')}\tilde{s}\sigma(\tilde{g},\theta')\delta^{(4)}(p'^\mu+q^\mu-p^\mu-q'^\mu),
\end{equation} so that we have
\begin{equation}\label{Ipq}
I=\int_\rth i(p,q)\sqrt{J(q)}dq.\end{equation}
We first translate (\ref{i}) into an expression involving the total and relative momentum
variables, $p'^\mu+q'^\mu$ and $p'^\mu-q'^\mu$ respectively. Define $u$ by $u(x) = 0$ if $x < 0$ and
$u(x) = 1$ if $x\geq 0$. Let $g'\eqdef g(p'^\mu,q'^\mu)$ and $s'\eqdef s(p'^\mu,q'^\mu).$ Then by the claim (7.5) of \cite{MR2728733}, we have
$$i(p,q)=\frac{1}{16p^0q^0}\int_{\rfo\times\rfo}d\Theta(p'^\mu,q'^\mu)\frac{e^{-q'^0/2}}{4\pi}\tilde{s}\sigma(\tilde{g},\theta')\delta^{(4)}(p'^\mu+q^\mu-p^\mu-q'^\mu),$$
where
$$d\Theta(p'^\mu,q'^\mu)\eqdef dp'^\mu dq'^\mu u(p'^0+q'^0)u(s'-4)\delta(s'-g'^2-4)\delta((p'^\mu+q'^\mu)(p'_\mu-q'_\mu)).$$
Thus we have lifted to an integral over $\rfo\times\rfo$ from one over $\rth\times\rth$.

Now we apply the change of variables $\bar{p}^\mu=p'^\mu+q'^\mu$ and $\bar{q}^\mu=p'^\mu-q'^\mu$. Then the Jacobian is 16. Since $q'^0=\frac{\bar{p}^0-\bar{q}^0}{2}$, we have
\begin{equation}\notag
i(p,q)= \frac{c'}{p^0q^0}\int_{\rfo\times\rfo}d\Theta(\bar{p}^\mu,\bar{q}^\mu)\tilde{s}\sigma(\tilde{g},\theta')\delta^{(4)}(q^\mu-p^\mu+\bar{q}^\mu)\exp\left(\frac{-\bar{p}^0+\bar{q}^0}{4}\right)
\end{equation} for some constant $c'>0$ (whose value can change from line to line), where
$$d\Theta(\bar{p}^\mu,\bar{q}^\mu)\eqdef d\bar{p}^\mu d\bar{q}^\mu u(\bar{p}^0)u(-\bar{p}^\mu\bar{p}_\mu-4)\delta(-\bar{p}^\mu\bar{p}_\mu-\bar{q}^\mu\bar{q}_\mu-4)\delta(\bar{p}^\mu\bar{q}_\mu).$$

We now carry out $\delta^{(4)}(q^\mu-p^\mu+\bar{q}^\mu)$ to obtain
\begin{equation}\notag
i(p,q)= \frac{c'}{p^0q^0}\int_{\rfo}d\Theta(\bar{p}^\mu)\tilde{s}\sigma(\tilde{g},\theta')\exp\left(\frac{-\bar{p}^0+p^0-q^0}{4}\right),
\end{equation}
where the measure $d\Theta(\bar{p}^\mu)$ is now equal to
$$d\Theta(\bar{p}^\mu)\eqdef d\bar{p}^\mu  u(\bar{p}^0)u(-\bar{p}^\mu\bar{p}_\mu-4)\delta(-\bar{p}^\mu\bar{p}_\mu-g^2-4)\delta(\bar{p}^\mu(p_\mu-q_\mu)).$$
Since $s=g^2+4,$ we have
\begin{multline*}
u(\bar{p}^0)\delta(-\bar{p}^\mu\bar{p}_\mu-g^2-4)=u(\bar{p}^0)\delta(-\bar{p}^\mu\bar{p}_\mu-s)\\=u(\bar{p}^0)\delta((\bar{p}^0)^2-|\bar{p}|^2-s)=\frac{\delta(\bar{p}^0-\sqrt{|\bar{p}|^2+s})}{2\sqrt{|\bar{p}|^2+s}}.
\end{multline*}
Then we carry out one integration using this delta function to obtain
\begin{multline}\notag
i(p,q)= \frac{c'}{2p^0q^0}\int_{\rth}\frac{d\bar{p}}{\bar{p}^0}u(-\bar{p}^\mu\bar{p}_\mu-4)\delta(\bar{p}^\mu(p_\mu-q_\mu))\tilde{s}\sigma(\tilde{g},\theta')\\\times\exp\left(\frac{-\sqrt{|\bar{p}|^2+s}+p^0-q^0}{4}\right),
\end{multline}
where $\bar{p}^0=\sqrt{|\bar{p}|^2+s}$. Using $s=g^2+4$ again, we have $$-\bar{p}^\mu\bar{p}_\mu-4=s-4=g^2\geq 0$$ to guarantee that $u(-\bar{p}^\mu\bar{p}_\mu-4)=1$. Thus
\begin{equation}\notag
i(p,q)= \frac{c'}{2p^0q^0}\exp\left(\frac{p^0-q^0}{4}\right)\int_{\rth}\frac{d\bar{p}}{\bar{p}^0}\delta(\bar{p}^\mu(p_\mu-q_\mu))\tilde{s}\sigma(\tilde{g},\theta')e^{\left(\frac{\bar{p}^\mu U_\mu}{4}\right)},
\end{equation}
where $\bar{p}^0=\sqrt{|\bar{p}|^2+s}$ and $U^\mu=(1,0,0,0)$.
We finish off our reduction by moving to a new Lorentz frame. We consider a Lorentz transformation $\Lambda$ which maps into the center-of-momentum system as 
$$
A_\nu\eqdef \Lambda^{\mu}{}_{\nu}(p_\mu+q_\mu) =(\sqrt{s},0,0,0),\hspace{10mm} B_\nu\eqdef -\Lambda^{\mu}{}_{\nu} (p_\mu-q_\mu)=(0,0,0,g).
$$
The explicit form of the matrix $\Lambda$ was given p. 593 of \cite{MR2728733}, and also in \cite{MR2707256}.  More precisely, we consider
\begin{equation}\label{eq.LT}
\Lambda=(\Lambda^{\mu}{}_\nu)=\left(\begin{array}{cccc}\frac{p^0+q^0}{\sqrt{s}}& -\frac{p_1+q_1}{\sqrt{s}} &-\frac{p_2+q_2}{\sqrt{s}} &-\frac{p_3+q_3}{\sqrt{s}}\\ \Lambda^1{}_0 &\Lambda^1{}_1&\Lambda^1{}_2&\Lambda^1{}_3\\ 0 & \frac{(p\times q)_1}{|p\times q|} &\frac{(p\times q)_2}{|p\times q|}&\frac{(p\times q)_3}{|p\times q|}\\ \frac{p^0-q^0}{g} &-\frac{p_1-q_1}{g}&-\frac{p_2-q_2}{g}&-\frac{p_3-q_3}{g}\end{array}\right),
\end{equation}
with the second row given by
$$\Lambda^1{}_0=\Lambda^1{}_0(p,q)=\frac{2|p\times q|}{g\sqrt{s}},$$ 
and for $i=1,2,3$ we have 
$$\Lambda^1{}_i=\Lambda^1{}_i(p,q)=\frac{2\left(p_i\{p^0+q^0p^\mu q_\mu\}+q_i\{q^0+p^0p^\mu q_\mu\}\right)}{g\sqrt{s}|p\times q|}.$$
Then, using this change of variables, we have
\begin{equation*}
\int_{\rth}\frac{d\bar{p}}{\bar{p}^0}\delta(\bar{p}^\mu(p_\mu-q_\mu))\tilde{s}\sigma(\tilde{g},\theta')e^{\left(\frac{\bar{p}^\mu U_\mu}{4}\right)}
=\int_{\rth}\frac{d\bar{p}}{\bar{p}^0}\delta(\bar{p}^\mu B_\mu)s_\Lambda\sigma(g_\Lambda,\theta_\Lambda)e^{\left(\frac{\bar{p}^\mu \bar{U}_\mu}{4}\right)}.
\end{equation*}
Note that $\frac{d\bar{p}}{\bar{p}^0}$ is Lorentz invariant.
Here $\bar{p}^0=\sqrt{|\bar{p}|^2+s}$ and $s_\Lambda$, $g_\Lambda\geq0$ are
$$
g^2_\Lambda\eqdef g^2-\frac{1}{2}A^\mu(\bar{p}_\mu-A_\mu)= g^2+\frac{1}{2}\sqrt{s}(\bar{p}^0-\sqrt{s}),
$$
where
\begin{equation}\label{slambda.def}
s_\Lambda\eqdef g^2_\Lambda+4,
\end{equation}
and
\begin{equation}\label{coslam}
\cos\theta_\Lambda\eqdef 2\frac{g^2}{g^2_\Lambda}-1.
\end{equation}
Also, $\bar{U}^\mu$ is defined as $\bar{U}^\mu= \left(\frac{p^0+q^0}{\sqrt{s}},\frac{2|p\times q|}{g\sqrt{s}},0,\frac{p^0-q^0}{g}\right)$.
We switch to polar coordinates in the form $$d\bar{p}=r^2 dr \sin\psi d\psi d\phi,\hspace{5mm} \bar{p}\eqdef r(\sin \psi \cos \phi, \sin \psi \sin \phi, \cos \psi).$$
Then we obtain $$\bar{p}^\mu B_\mu = gr\cos\psi.$$
Then the integral $i(p,q)$ is now equal to
\begin{multline}\notag
i(p,q)= \frac{c'}{2p^0q^0}\exp\left(\frac{p^0-q^0}{4}\right)\int_{0}^{2\pi}d\phi \int_{0}^{\pi}d\psi \sin\psi \\\times\int_{0}^{\infty} \frac{r^2dr}{\sqrt{r^2+s}}\delta(gr\cos\psi)s_\Lambda\sigma(g_\Lambda,\theta_\Lambda) e^{\left(\frac{\bar{p}^\mu \bar{U}_\mu}{4}\right)}.
\end{multline}
We evaluate the last delta function at $\psi = \pi/2$ to write $i(p, q)$ as
\begin{multline}\notag
i(p,q)= \frac{c'}{2gp^0q^0}\exp\left(\frac{p^0-q^0}{4}\right)  \\\times\int_{0}^{2\pi}d\phi \int_{0}^{\infty} \frac{rdr}{\sqrt{r^2+s}}s_\Lambda\sigma(g_\Lambda,\theta_\Lambda)\exp\left(-\bar{p}^0\frac{p^0+q^0}{4\sqrt{s}}+\frac{|p\times q|}{2g\sqrt{s}}r\cos\phi\right).
\end{multline}
By using the modified Bessel function of index zero given by \eqref{bessel0} we have
\begin{multline}\label{finali}
i(p,q)= \frac{c'}{2gp^0q^0}\exp\left(\frac{p^0-q^0}{4}\right)  \\\times \int_{0}^{\infty} \frac{rdr}{\sqrt{r^2+s}}s_\Lambda\sigma(g_\Lambda,\theta_\Lambda)\exp\left(-\frac{p^0+q^0}{4\sqrt{s}}\sqrt{r^2+s}\right)I_0\left(\frac{|p\times q|}{2g\sqrt{s}}r\right).
\end{multline}
Now $g_\Lambda\geq0$ is given by
\begin{equation}\label{g2}
g^2_\Lambda = g^2+\frac{1}{2}\sqrt{s}(\sqrt{r^2+s}-\sqrt{s}),
\end{equation}
with \eqref{slambda.def} and \eqref{coslam}.
So by \eqref{Ipq} we obtain a new representation of our gain term
\begin{multline}\label{final gain}
I=I_{gain}=\frac{c'}{p^0}e^{\frac{p^0}{4}}\int_{\rth}\frac{dq}{q^0}\frac{e^{-\frac{3}{4}q^0}}{g} \int_{0}^{\infty} \frac{rdr}{\sqrt{r^2+s}}s_\Lambda\sigma(g_\Lambda,\theta_\Lambda)\\\times \exp\left(-\frac{p^0+q^0}{4\sqrt{s}}\sqrt{r^2+s}\right)I_0\left(\frac{|p\times q|}{2g\sqrt{s}}r\right),
\end{multline}
We recall that $I_{loss}=0$. We will now find a different expression than $I_{loss}$ which is also equal to zero, that will provide suitable cancellation for the term in (\ref{final gain}) when we no longer assume that $\int_{\mathbb{S}^2} d\omega\  |\sigma_0(\cos\theta)| <\infty$ and $\int_{\mathbb{S}^2} d\omega\  \sigma_0(\cos\theta)=0$.

To this end, we recall the definitions \eqref{g2} and \eqref{coslam}.  This yields
\begin{equation}\label{newcos}
\cos\theta_\Lambda=\frac{2g^2}{g^2_\Lambda}-1=\frac{g^2-\frac{1}{2}\sqrt{s}(\sqrt{r^2+s}-\sqrt{s})}{g^2+\frac{1}{2}\sqrt{s}(\sqrt{r^2+s}-\sqrt{s})},
\end{equation}
using \eqref{g2} above.  Then we have
$$\frac{dg^2_\Lambda}{dr}= \frac{\sqrt{s}r}{2\sqrt{r^2+s}}.$$
By differentiating $\cos\theta_\Lambda$ with respect to $r$, we have
\begin{multline*}
\frac{d(\cos\theta_\Lambda)}{dr}
=\frac{d}{dr}\frac{g^2-\frac{1}{2}\sqrt{s}(\sqrt{r^2+s}-\sqrt{s})}{g^2_\Lambda}\\=\frac{-\frac{1}{2}\sqrt{s}\frac{2r}{2\sqrt{r^2+s}}g_\Lambda^2-(g^2-\frac{1}{2}\sqrt{s}(\sqrt{r^2+s}-\sqrt{s}))\frac{d g^2_\Lambda}{dr}  }{g^4_\Lambda}\\
=-\frac{1}{2}\sqrt{s}\frac{r}{\sqrt{r^2+s}g^2_\Lambda}-\frac{(g^2-\frac{1}{2}\sqrt{s}(\sqrt{r^2+s}-\sqrt{s})) }{g^4_\Lambda} \frac{\sqrt{s}r}{2\sqrt{r^2+s}}\\
=\frac{\sqrt{s}r}{2g^4_\Lambda\sqrt{r^2+s}}\left(-g^2_\Lambda -(g^2-\frac{1}{2}\sqrt{s}(\sqrt{r^2+s}-\sqrt{s}))\right)\\
=\frac{\sqrt{s}r}{2g^4_\Lambda\sqrt{r^2+s}}\left(-2g^2\right)=-\frac{g^2\sqrt{s}r}{g^4_\Lambda\sqrt{r^2+s}}.
\end{multline*}Therefore,
$$\frac{d(\cos\theta_\Lambda)}{dr}=-\frac{g^2\sqrt{s}r}{g^4_\Lambda\sqrt{r^2+s}}.$$
Since we have assumed that $$\int_{-1}^1d(\cos\theta_\Lambda) \sigma_0(\cos\theta_\Lambda)=0,$$ we further have
$$
\int_{0}^{\infty}dr \frac{g^2\sqrt{s}r}{g^4_\Lambda\sqrt{r^2+s}}\sigma_0(\cos\theta_\Lambda)=0,
$$ as $\cos\theta_\Lambda = 1\text{ and }-1$ correspond to $r=0$ and $r=\infty$ respectively, by \eqref{newcos}.  Thus we obtain
\begin{equation}
\notag
\frac{c'}{p^0}e^{\frac{p^0}{4}}\int_{\rth}\frac{dq}{q^0}\frac{e^{-\frac{3}{4}q^0}}{g} \int_{0}^\infty \frac{rdr}{\sqrt{r^2+s}}s\Phi(g)\sigma_0(\cos \theta_\Lambda)\frac{g^4}{g^4_\Lambda}\exp\left(-\frac{p^0+q^0}{4\sqrt{s}}\sqrt{s}\right)=0.
\end{equation}
Subtracting this zero integral from (\ref{final gain}), we obtain
\begin{multline}\label{dual}
I=\frac{c'}{p^0}e^{\frac{p^0}{4}}\int_{\rth}\frac{dq}{q^0}\frac{e^{-\frac{3}{4}q^0}}{g}\int_{0}^\infty \frac{rdr}{\sqrt{r^2+s}}s_\Lambda\sigma(g_\Lambda,\theta_\Lambda)\\\times\Big[\exp\left(-\frac{p^0+q^0}{4\sqrt{s}}\sqrt{r^2+s}\right)I_0\left(\frac{|p\times q|}{2g\sqrt{s}}r\right)\\ -\exp\left(-\frac{p^0+q^0}{4}\right)\frac{s\Phi(g)g^4}{s_\Lambda  \Phi(g_\Lambda)g^4_\Lambda}\Big].
\end{multline}
This is equal to the original integral $I=-\tilde{\zeta}(p)$ when the mean value of $\sigma_0$ is zero.

 We also note that \eqref{dual} also holds for \eqref{tildezeta} even when the mean value of $\sigma_0$ is not zero.  
	Suppose that $\int_{\mathbb{S}^2} d w \hspace{1mm} |\sigma_0(\theta)| <\infty$ and that $\int_{\mathbb{S}^2} d w \hspace{1mm} \sigma_0(\theta) = 2\pi c_0\neq0$.
	Define 
	$$
	\sigma_0^\epsilon(\theta) = \sigma_0(\theta)-1_{[1-\epsilon,1]}(\cos\theta)\int_{-1}^1dt'\frac{\sigma_0(t')}{\epsilon}. 
	$$
	Then, we have $\int_{-1}^1\sigma_0^\epsilon(\theta)d(\cos\theta)=0$ vanishing on $\omega\in \mathbb{S}^2$.
	Now, define 
	\begin{align*}
	&\tilde{\zeta}^\epsilon(p)
	=\int_{\rth\times\mathbb{S}^2} v_{\text{\o}} \Phi(g)\sigma^\epsilon_0(\theta) \sqrt{J(q)}\left(\sqrt{J(q)}-\sqrt{J(q')}\right)dqd\omega.
	\end{align*} Then, also using  \eqref{tildezeta}, we have
	\begin{multline}
	\label{notmeanzero}
	\left|\tilde{\zeta}(p) -\tilde{\zeta}^\epsilon(p) \right|
	\\
	=\bigg|c_0 \int_{\rth\times\mathbb{S}^2} v_{\text{\o}} \Phi(g) \sqrt{J(q)}\left(\sqrt{J(q)}-\sqrt{J(q')}\right) \frac{1_{[1-\epsilon,1]}(\cos\theta)}{\epsilon} dqd\omega\bigg|.
	\end{multline}
 If $\cos\theta=1$, by Remark \ref{angle.remark} and e.g. \eqref{cosine.angle.formula}, \eqref{gbar} and \eqref{conservation}  we have $p'^\mu=p^\mu$ and $q'^\mu=q^\mu$. 
	Thus, as $\epsilon \rightarrow 0$, the difference term in   (\ref{notmeanzero})$\rightarrow 0$ as $\sqrt{J(q)}-\sqrt{J(q')}$ has a higher order cancellation and hence the integrand vanishes on the set $\cos\theta=1$. 
 By the higher-order cancellation, an additional cutoff argument shows that the identity \eqref{dual} holds for the noncutoff kernel $\sigma_0$ from \eqref{angassumption}.

\subsection{First representation of $\tilde{\zeta}$}
We will now further split $\tilde{\zeta}=\zeta_0+\zetaL$. From \eqref{dual} for simplicity we write
$$
-I=\frac{c'}{p^0}e^{\frac{p^0}{4}}\int_{\rth}\frac{dq}{q^0}\frac{e^{-\frac{3}{4}q^0}}{g}\newK(p,q),
$$
where
\begin{multline}\notag
\newK(p,q) \eqdef \int_{0}^\infty \frac{rdr}{\sqrt{r^2+s}}s_\Lambda\sigma(g_\Lambda,\theta_\Lambda)
\\
\times
\Big[\exp\left(-\frac{p^0+q^0}{4}\right)\frac{s\Phi(g)g^4}{s_\Lambda  \Phi(g_\Lambda)g^4_\Lambda} - \exp\left(-\frac{p^0+q^0}{4\sqrt{s}}\sqrt{r^2+s}\right)I_0\left(\frac{|p\times q|}{2g\sqrt{s}}r\right) \Big].
\end{multline}
Notice that both terms of the integral converge for large $r\ge 1$.  We further split
\begin{multline}\notag
\exp\left(-\frac{p^0+q^0}{4}\right)\frac{s\Phi(g)g^4}{s_\Lambda  \Phi(g_\Lambda)g^4_\Lambda}-\exp\left(-\frac{p^0+q^0}{4\sqrt{s}}\sqrt{r^2+s}\right)I_0\left(\frac{|p\times q|}{2g\sqrt{s}}r\right)
\\
=\left(\exp\left(-\frac{p^0+q^0}{4}\right) - \exp\left(-\frac{p^0+q^0}{4\sqrt{s}}\sqrt{r^2+s}\right)I_0\left(\frac{|p\times q|}{2g\sqrt{s}}r\right)\right) \frac{s\Phi(g)g^4}{s_\Lambda  \Phi(g_\Lambda)g^4_\Lambda}\\
+\exp\left(-\frac{p^0+q^0}{4\sqrt{s}}\sqrt{r^2+s}\right) I_0\left(\frac{|p\times q|}{2g\sqrt{s}}r\right)
\left( \frac{s\Phi(g)g^4}{s_\Lambda  \Phi(g_\Lambda)g^4_\Lambda} -1\right).
\end{multline}
This motivates the following splitting of $\tilde{\zeta}=\zeta_0+\zetaL$ with
\begin{multline}\label{zeta0B.appendix}
\zeta_0 \eqdef \frac{c'}{p^0}e^{\frac{p^0}{4}}\int_{\rth}\frac{dq}{q^0}\frac{e^{-\frac{3}{4}q^0}}{g}\int_{0}^\infty \frac{rdr}{\sqrt{r^2+s}}s_\Lambda\sigma(g_\Lambda,\theta_\Lambda) \frac{s\Phi(g)g^4}{s_\Lambda  \Phi(g_\Lambda)g^4_\Lambda}
\\
\times
\Big[\exp\left(-\frac{p^0+q^0}{4}\right) - \exp\left(-\frac{p^0+q^0}{4\sqrt{s}}\sqrt{r^2+s}\right)I_0\left(\frac{|p\times q|}{2g\sqrt{s}}r\right) \Big],
\end{multline}and
\begin{multline}\label{zetaLB}
\zetaL \eqdef \frac{c'}{p^0}e^{\frac{p^0}{4}}\int_{\rth}\frac{dq}{q^0}\frac{e^{-\frac{3}{4}q^0}}{g}\int_{0}^\infty \frac{rdr}{\sqrt{r^2+s}}s_\Lambda\sigma(g_\Lambda,\theta_\Lambda)   \\
\times
\exp\left(-\frac{p^0+q^0}{4\sqrt{s}}\sqrt{r^2+s}\right) I_0\left(\frac{|p\times q|}{2g\sqrt{s}}r\right)
\left( \frac{s\Phi(g)g^4}{s_\Lambda  \Phi(g_\Lambda)g^4_\Lambda} -1\right).
\end{multline}
This completes the derivation of our first representation of $\tilde{\zeta}(p)$.

\subsection{Derivation of an alternative representation of $\tilde{\zeta}(p)$}\label{sec:alternative.deriv}
For a fixed $p\in\rth$, we would like to have an alternative representation of \eqref{tildezeta}:
\begin{equation}\label{Irepresentation}
 -\tilde{\zeta}(p)=	\int_{\rth\times\mathbb{S}^2} v_{\text{\o}} \sigma(g,\theta) \sqrt{J(q)}\left(\sqrt{J(q')}-\sqrt{J(q)}\right)dqd\omega\eqdef I_{gain}-I_{loss}.
\end{equation}
Then exactly as previously we can derive \eqref{final gain} for the term $I_{gain}$.  We will now find an alternative expression for $I_{loss}$ that will provide suitable cancellation for the term in (\ref{final gain}). 

To this end, using the definition of $I_{loss}$ from \eqref{Irepresentation} yields
\begin{equation}\label{Ilossrepresentation}
I_{loss}=\int_{\rth}dq\int_{\mathbb{S}^2}d\omega\  v_{\text{\o}} \sigma(g,\theta)J(q).
\end{equation}Since the gain term representation \eqref{Igainrepresentation} results in \eqref{i} and \eqref{Ipq}, the loss term \eqref{Ilossrepresentation} would yield
\begin{equation}\notag
I_{loss}=\int_{\rth}i_{loss}(p,q)dq,
\end{equation}where 
\begin{equation}\label{iloss}
i_{loss}(p,q)\eqdef \frac{1}{p^0q^0} \int_{\rth}\frac{dp'}{p'^0}\int_{\rth}\frac{dq'}{q'^0}J(q')\tilde{s}\sigma(\tilde{g},\theta')\delta^{(4)}(p'^\mu+q^\mu-p^\mu-q'^\mu),
\end{equation} 
by following the same argument between \eqref{Igainrepresentation} and \eqref{Ipq}. Note that we have exchanged $q$ and $q'$ variables in the procedure. Then we can easily see that the only difference between \eqref{i} and \eqref{iloss} is the power on the term $J(q')$; i.e., the power on $J(q')$ in $i_{loss}(p,q)$ is twice of that in $i(p,q)$. Therefore, the same derivation results in the new representation of the loss term similar to  \eqref{finali}:
\begin{multline}\label{finaliloss}
i_{loss}(p,q)= \frac{c'}{gp^0q^0}\exp\left(\frac{p^0-q^0}{2}\right)  \\\times \int_{0}^{\infty} \frac{rdr}{\sqrt{r^2+s}}s_\Lambda\sigma(g_\Lambda,\theta_\Lambda)\exp\left(-\frac{p^0+q^0}{2\sqrt{s}}\sqrt{r^2+s}\right)I_0\left(\frac{|p\times q|}{g\sqrt{s}}r\right).
\end{multline}
In particular we have
\begin{multline}\label{final loss}
I_{loss}=\frac{c'}{p^0}e^{\frac{p^0}{2}}\int_{\rth}\frac{dq}{q^0}\frac{e^{-\frac{1}{2}q^0}}{g} \int_{0}^{\infty} \frac{rdr}{\sqrt{r^2+s}}s_\Lambda\sigma(g_\Lambda,\theta_\Lambda)\\\times \exp\left(-\frac{p^0+q^0}{2\sqrt{s}}\sqrt{r^2+s}\right)I_0\left(\frac{|p\times q|}{g\sqrt{s}}r\right).
\end{multline}
Subtracting this integral from (\ref{final gain}), we obtain
\begin{multline}\notag
I=\frac{c'}{p^0}e^{\frac{p^0}{4}}\int_{\rth}\frac{dq}{q^0}\frac{e^{-\frac{3}{4}q^0}}{g}\int_{0}^\infty \frac{rdr}{\sqrt{r^2+s}}s_\Lambda\sigma(g_\Lambda,\theta_\Lambda)\\\times\Big[\exp\left(-\frac{p^0+q^0}{4\sqrt{s}}\sqrt{r^2+s}\right)I_0\left(\frac{|p\times q|}{2g\sqrt{s}}r\right)\\ -\exp\left(\frac{p^0+q^0}{4}\right)\exp\left(-\frac{p^0+q^0}{2\sqrt{s}}\sqrt{r^2+s}\right)I_0\left(\frac{|p\times q|}{g\sqrt{s}}r\right)\Big].
\end{multline}
This is equal to the original integral $I=-\tilde{\zeta}(p)$. 
The representation above also holds when $\sigma_0$ does not have mean zero or does not have a bounded integral, as we discussed in \eqref{notmeanzero}.

As in \eqref{zeta0By} and \eqref{zetaLBy}, we take the change of variables $r\mapsto y=\frac{r}{\sqrt{s}}$ in $I$ above. 
Then, we can write $\tilde{\zeta}$ as follows
\begin{multline}\label{eq:tildezetanew1.appendix}
\tilde{\zeta}(p)=\frac{c'}{\pi p^0}\int_{\rth}\frac{dq}{q^0}\frac{e^{-q^0}\sqrt{s}}{g}\int_{0}^\infty \frac{ydy}{\sqrt{y^2+1}}s_\Lambda\sigma(g_\Lambda,\theta_\Lambda)\int_0^{\pi} d\phi\\\times\Big[\exp (2l-2l\sqrt{y^2+1}+2jy\cos\phi)-\exp(l-l\sqrt{y^2+1} +jy\cos\phi)\Big],
\end{multline}
where we use the notations \eqref{lj} and \eqref{g2y.variable}. This completes the derivation.

\subsection{Proofs of the pointwise estimates}\label{sec:pointwiseProof}
We will now give the proofs of Lemma \ref{lem:useful.ests} and Lemma \ref{lem:integral.ests}.

\begin{proof}[Proof of Lemma \ref{lem:useful.ests}]
The proof of \eqref{s.ge.g2} is direct, and \eqref{s.le.pq} follows from \eqref{s} and the Cauchy-Schwartz inequality.  Then \eqref{g.ge.lower}, \eqref{g.ge.2lower} and \eqref{g.le.upper} follow from \eqref{g.ineq.sharp}.  For \eqref{p0q0.le.pq} notice that that
$$
p^0 - q^0 = \frac{|p|^2 - |q|^2}{p^0 + q^0} = \frac{(p - q) \cdot (p+q)}{p^0 + q^0} 
\le |p - q|.
$$
Then \eqref{p0.plus.q0.le.p0q0} is automatic.

Now using \eqref{lj} then equation \eqref{j.le.l} follows from \eqref{g.ge.2lower}.  And \eqref{l.upper.ineq} is automatic.  The proof of \eqref{l2j2} requires some development and is from \cite{GS3}.  Now using \eqref{lj} we have
$$
    l^2-j^2 = \left(\frac{p^0+q^0}{4}\right)^2-\left(\frac{|p\times q|}{2g}\right)^2
    =\frac{(p^0+q^0)^2g^2-4|p\times q|^2}{16g^2}.$$
    By the definition of $g$ in \eqref{g}, we have
\begin{multline*}
    (p^0+q^0)^2g^2-4|p\times q|^2
    =(p^0+q^0)^2(-2-2p^\mu q_\mu)-4|p\times q|^2\\
    =(2+|p|^2+|q|^2+2p^0q^0)(-2+2p^0q^0-2p\cdot q)-4|p\times q|^2\\
    =(2+|p|^2+|q|^2+2p^0q^0)(-2-|p|^2-|q|^2+2p^0q^0+|p-q|^2)-4|p\times q|^2\\
    =(2p^0q^0)^2-(2+|p|^2+|q|^2)^2 +(2+|p|^2+|q|^2+2p^0q^0)|p-q|^2-4|p\times q|^2\\
    =(2p^0q^0)^2-(2+|p|^2+|q|^2)^2 +(p^0+q^0)^2|p-q|^2-4|p\times q|^2.
\end{multline*}
We calculate that
\begin{multline*}
    (2p^0q^0)^2-4|p\times q|^2=4+4|p|^2|q|^2+4|p|^2+4|q|^2-4|p\times q|^2\\
    =4+4(p\cdot q)^2+4|p|^2+4|q|^2,
\end{multline*}
and$$
    (2+|p|^2+|q|^2)^2=4+|p|^4+|q|^4+4|p|^2+4|q|^2+2|p|^2|q|^2.$$
Thus, using also \eqref{s}, we have
\begin{multline*}
    (2p^0q^0)^2-(2+|p|^2+|q|^2)^2 +(p^0+q^0)^2|p-q|^2-4|p\times q|^2\\
    =(p^0+q^0)^2|p-q|^2-|p|^4-|q|^4+4(p\cdot q)^2-2|p|^2|q|^2\\
    =(p^0+q^0)^2|p-q|^2-(|p|^2+|q|^2)^2 +4(p\cdot q)^2\\
    =(p^0+q^0)^2|p-q|^2 -(|p|^2+|q|^2+2p\cdot q)(|p|^2+|q|^2-2p\cdot q)\\
    =(p^0+q^0)^2|p-q|^2-|p+q|^2|p-q|^2
    =s|p-q|^2.
\end{multline*}
Therefore, we have \eqref{l2j2}.   Then \eqref{l2j2size} follows from \eqref{l2j2} and \eqref{s.ge.g2}.

We will now prove \eqref{ineq.gL.here}.  The upper bound of $g_\Lambda^2$ in \eqref{ineq.gL.here} follows from \eqref{g2y.variable} with \eqref{s.ge.g2}.  The lower bound of $g_\Lambda^2$ in \eqref{ineq.gL.here} follows from \eqref{glambda.calc} and \eqref{s.ge.g2}.  
\end{proof}

\begin{proof}[Proof of Lemma \ref{lem:integral.ests}]
We will start with \eqref{smally.lemma}.  From \eqref{bessel0} and \eqref{int.Kgamma} we have 
 $$
|\bar{K}_\gamma(l,j)|  \lesssim
\max_{0\leq x\leq 1}\exp(-l\sqrt{x^2+1}+jx).
$$ 
The maximum of the function $h(x)\eqdef -l\sqrt{x^2+1}+jx$ occurs at $x=0$, $x=1$, or $x=x_0=\frac{j}{\sqrt{l^2-j^2}}$ where $h'(x_0)=0$. Note that $ h(x_0)=-\sqrt{l^2-j^2}.$ When $x=1$, we have $$h(1)=-\sqrt{2}l+j\leq -\sqrt{l^2-j^2}.$$ Thus, we conclude \eqref{max.bound} and \eqref{smally.lemma}.
	
Then \eqref{J2.special} is a known integral that can be calculated exactly \cite{Gradshteyn:1702455} as  \eqref{J2.lemma}.  Further  \eqref{k2lj.lemma} is calculated during the proof of Corollary 2  in \cite[Corollary 2, pp.~323]{GS3}.  In particular we can obtain \eqref{k2lj.lemma} from $\tilde{K}_2(l,j)= \partial_l^2 J_2(l,j)$.
\end{proof}

\providecommand{\bysame}{\leavevmode\hbox to3em{\hrulefill}\thinspace}
\providecommand{\href}[2]{#2}

\end{document}